\documentclass[a4paper,fontsize=10pt]{scrartcl}
\usepackage[utf8]{inputenc}
\usepackage[T1]{fontenc}
\usepackage[english]{babel}
\usepackage{amsmath,amsthm,amssymb,amsxtra}
\usepackage[a4paper]{geometry}
\usepackage{xcolor}
\usepackage{enumitem}

\usepackage{newunicodechar}
\newunicodechar{α}{\alpha}
\newunicodechar{β}{\beta}
\newunicodechar{γ}{\gamma}
\newunicodechar{δ}{\delta}
\newunicodechar{ε}{\varepsilon}
\newunicodechar{ϵ}{\epsilon}
\newunicodechar{ζ}{\zeta}
\newunicodechar{η}{\eta}
\newunicodechar{θ}{\theta}
\newunicodechar{ϑ}{\vartheta}
\newunicodechar{ι}{\iota}
\newunicodechar{κ}{\kappa}
\newunicodechar{λ}{\lambda}
\newunicodechar{μ}{\mu}
\newunicodechar{ν}{\nu}
\newunicodechar{ξ}{\xi}
\newunicodechar{ο}{\omikron}
\newunicodechar{π}{\pi}
\newunicodechar{ρ}{\rho}
\newunicodechar{σ}{\sigma}
\newunicodechar{τ}{\tau}
\newunicodechar{υ}{\upsilon}
\newunicodechar{φ}{\varphi}
\newunicodechar{ϕ}{\phi}
\newunicodechar{χ}{\chi}
\newunicodechar{ψ}{\psi}
\newunicodechar{ω}{\omega}
\newunicodechar{Γ}{\Gamma}
\newunicodechar{Δ}{\Delta}
\newunicodechar{Θ}{\Theta}
\newunicodechar{Λ}{\Lambda}
\newunicodechar{Ξ}{\Xi}
\newunicodechar{Π}{\Pi}
\newunicodechar{Σ}{\Sigma}
\newunicodechar{Φ}{\Phi}
\newunicodechar{Ψ}{\Psi}
\newunicodechar{Ω}{\Omega}
\newunicodechar{ℝ}{\mathbb{R}}
\newunicodechar{ℤ}{\mathbb{Z}}
\newunicodechar{ℕ}{\mathbb{N}}
\newunicodechar{ℚ}{\mathbb{Q}}
\newunicodechar{ℂ}{\mathbb{C}}
\newunicodechar{→}{\to}
\newunicodechar{∞}{\infty}
\newunicodechar{∈}{\in}
\newunicodechar{∩}{\cap}
\newunicodechar{∪}{\cup}
\newunicodechar{⊂}{\subset}
\newunicodechar{×}{\times}
\newunicodechar{∇}{\nabla}
\newunicodechar{∅}{\emptyset}
\newunicodechar{∀}{\forall}
\newunicodechar{∃}{\exists}
\newunicodechar{∂}{\partial}
\newunicodechar{⇐}{\Leftarrow}
\newunicodechar{⇔}{\LeftRightarrow}
\newunicodechar{⇒}{\Rightarrow}
\newunicodechar{↦}{\mapsto}
\newunicodechar{⊥}{\perp}

\newcommand{\beenum}{\begin{enumerate}[label={\roman*)},leftmargin=0em, itemindent=0.5cm, topsep=0cm, itemsep=0cm]}

\newcommand{\delny}{\partial_\nu}
\newcommand{\amrand}{\big\rvert_{\partial\Omega}}
\newcommand{\Lom}[1]{L^{#1}(\Omega)}
\newcommand{\Lmme}{L^{\max\{1,m-1\}}(\Om)}
\newcommand{\Wom}[2]{W^{#1,#2}(\Omega)}
\newcommand{\Li}{L^{∞}}
\newcommand{\Liom}{L^{∞}(Ω)}
\newcommand{\Om}{Ω}
\newcommand{\Ombar}{\overline{\Omega}}
\newcommand{\dOm}{\partial \Omega}
\newcommand{\LT}[2]{L^{#1}((0,T);#2)}
\newcommand{\norm}[2][]{\left\lVert#2\right\rVert_{#1}}
\newcommand{\normm}[2]{\left\lVert#2\right\rVert_{#1}}
\newcommand{\kl}[1]{\left(#1\right)}
\newcommand{\qtilde}{\widetilde{q}}
\newcommand{\utilde}{\widetilde{u}}
\newcommand{\vtilde}{\widetilde{v}}
\newcommand{\wtilde}{\widetilde{w}}
\newcommand{\Ttilde}{\widetilde{T}}
\newcommand{\set}[1]{\left\{#1\right\}}
\newcommand{\ds}{\mathrm{d}s}
\newcommand{\dsigma}{\mathrm{d}\sigma}
\newcommand{\restr}[1]{\big\vert_{#1}}
\newcommand{\intninf}{\int_0^{∞}}
\newcommand{\intnt}{\int_0^t}
\newcommand{\intnT}{\int_0^T}
\newcommand{\uhat}{\widehat{u}}
\newcommand{\ubar}{\overline{u}}
\newcommand{\vbar}{\overline{v}}
\newcommand{\weakstarto}{\stackrel{*}{\rightharpoonup}}
\newcommand{\wto}{\rightharpoonup}

\newcommand{\upto}{\nearrow}
\newcommand{\Tmax}{T_{max}}
\newcommand{\io}{\int_{Ω}}
\newcommand{\ddt}{\frac{\mathrm{d}}{\mathrm{d}t}}
\newcommand{\nn}{\nonumber}
\newcommand{\GNI}{Gagliardo-Nirenberg inequality}
\newcommand{\YI}{Young's inequality}
\newcommand{\na}{∇}
\newcommand{\calS}{\mathcal{S}}
\newcommand{\downto}{\searrow}
\newcommand{\sub}{\subset}
\newcommand{\Cdm}{\mathcal{C}_{δ,m}}
\newcommand{\fatntmax}{\qquad \text{for all } t\in(0,\Tmax)}
\newcommand{\usualspace}[1][\Tmax]{C^0(\Ombar\times[0,#1))\cap C^{2,1}(\Ombar\times(0,#1))}
\newcommand{\cptembeddedinto}{\stackrel{cpt}{\hookrightarrow}}
\newcommand{\embeddedinto}{\hookrightarrow}

\newcommand{\ID}{\mathcal{ID}_K}
\accentedsymbol{\Dbar}{\Bar{D}}
\accentedsymbol{\Dbarbar}{\Bar{\Bar{D}}}
\newcommand{\betr}[1]{\left\lvert#1\right\rvert}
\newcommand{\une}{u_{0,ε}}
\newcommand{\vne}{v_{0,ε}}
\newcommand{\epstilde}{\widetilde{ε}}

\newtheorem{theorem}{Theorem}[section]
\newtheorem{lemma}[theorem]{Lemma}
\theoremstyle{definition}
\newtheorem{defn}[theorem]{Definition}
\title{Locally bounded global solutions to a chemotaxis consumption model with singular sensitivity and nonlinear diffusion}
\author{Johannes Lankeit\thanks{Institut f\"ur Mathematik, Universit\"at Paderborn, Warburger Str.100, 33098 Paderborn, Germany; email: johannes.lankeit@math.uni-paderborn.de}}
\date{}

\begin{document}
\maketitle
\begin{abstract}
\noindent We show the existence of locally bounded global solutions to the chemotaxis system 
\[
 \begin{cases} 
  u_t=\nabla\cdot(D(u)\nabla u) - \nabla\cdot(\frac{u}v\nabla v)&\text{in }\Omega\times(0,\infty)\\
  v_t=\Delta v - uv&\text{in } \Omega\times(0,\infty)\\
  \partial_\nu u = \partial_\nu v = 0 & \text{in } \partial\Omega\times(0,\infty)\\
  u(\cdot,0)=u_0, v(\cdot,0)=v_0&\text{in } \Omega
 \end{cases}
\]
 in smooth bounded domains $\Om\subℝ^N$, $N\geq2$, for $D(u)\geq δu^{m-1}$ with some $δ>0$, provided that $m>1+\frac N4$. 
{\bf Math Subject Classification (2010):} 35K55, 35A01, 35K65, 92C17
\end{abstract}

\section{Introduction}
Even simple, small organisms can exhibit comparatively complex and macroscopically apparent collective behaviour. Bacteria of the species \textit{E. coli}, for example, when  set in a capillary tube featuring a gradient of nutrient concentration form bands that are visible to the naked eye and migrate with constant speed. 
Following experimental works of Adler (see e.g. \cite{adler1966chemotaxis,adler_dahl}), in 1971 Keller and Segel (\cite{KS71_traveling}) introduced a phenomenological model to capture this kind of behaviour, a prototypical version of which is given by  

\begin{equation}\label{intro:system}
 \begin{cases} 
  u_t=\nabla\cdot(D(u)\nabla u) - \nabla\cdot(\frac{u}v\nabla v)&\text{in }\Omega\times(0,\infty)\\
  v_t=\Delta v - uv&\text{in } \Omega\times(0,\infty)\\
 \end{cases}
\end{equation}
with $D(u)\equiv 1$. Herein, $u$ represents the density of bacteria and $v$ is used to denote the concentration of the nutrient. 
In the model in \cite{KS71_traveling}, the diffusion coefficient $D(u)$ is supposed to be constant, thus leading to the typical effect of linear diffusion which causes any population to spread with infinite speed of propagation. In order to avoid this (biologically clearly unrealistic) behaviour, it might be desirable to allow for diffusion of porous medium type (i.e. $D(u)=u^{m-1}$), cf. also \cite[p. 1665]{BBTW}.  

Nevertheless, starting with \cite{KS71_traveling}, the model with linear diffusion has successfully been employed to find travelling wave solutions (see e.g. the overview in \cite{wang_review} and references cited therein) and also their stability has been investigated (\cite{li_li_wang},\cite{nagai_ikeda}). 


In spite of the rich literature concerned with travelling wave solutions (for such solutions to related systems see also \cite{mei_peng_wang},\cite{meyries}, \cite{li_wang_steadily}, or \cite{horstmann_stevens}, \cite{meyries_rademacher_siero}), little is known about existence of solutions for more general initial data (see below).\\ 

The difficulty lies in the hazardous combination of the consumptive effect of the second equation on the nutrient concentration with the singular chemotactic sensitivity in the first: While the second equation compels $v$ to shrink, it is the cross-diffusive contribution of the chemotaxis term that seeks to enlarge the solutions to \eqref{intro:system}. And it is this very term that is furnished with a large coefficient whenever $v$ becomes small.\\

For a moment leaving aside the logarithmic shape of the sensitivity in \mbox{$\na\cdot(\frac uv\na v)=\na\cdot (u\na \log v)$}, we are led to the system 
\begin{equation}\label{eq:nonlogsens}
 \begin{cases}
  u_t=\Delta u - \na \cdot (u\na v),\\
  v_t=\Delta v - uv, 
 \end{cases}
\end{equation}
which also appears as part of chemotaxis fluid systems intensively studied during the past six years. (The interested reader can consult the introduction of \cite{lankeit_ctfluidlogsource}.)
Even in \eqref{eq:nonlogsens}, global existence of classical solutions is not yet known, apart from 2-dimensional settings (\cite{win_ctfluid}) or under smallness conditions on $v_0$ (\cite{tao11_boundedness}).\\ 

Although the mathematical difficulty in treating the system vastly increases when a logarithmic sensitivity is included, this form is important. Not only is it needed for the emergence of travelling waves (\cite{KS71_traveling,keller_odell,schwetlick}), there are also models giving a detailled mechanistic basis (\cite{xue}) and 
experimental evidence asserting this form (\cite{kalinin_jiang_tu_wu}).\\

In those Keller-Segel models (cf. \cite{horstmannI,hillen_painter,BBTW}) where $v$ does not stand for a nutrient to be consumed but a signalling substance produced by the bacteria themselves, i.e. the evolution is governed by 
\[
 \begin{cases}
  u_t=\Delta u - χ\na\cdot (\frac{u}v\na v),\\
  v_t=\Delta v - v +u, 
 \end{cases}
\]
the singularity in the sensitivity function is mitigated by $v$ tending to stay away from $0$ thanks to the production term in the second equation. 
(For this system, global solutions are known to exist if $χ$ is sufficiently small, where the precise condition depends on the dimension as well as on whether classical (\cite{lankeit_singularsensitivity,win_singularglobal,biler99}) or weak solutions (\cite{stinner_win,win_singularglobal}) are considered and on radial symmetry of initial data (\cite{biler99,nagai_senba98}); but for large $χ$ also blow-up may occur in the corresponding parabolic-elliptic system (\cite{nagai_senba98}).)
The proof of boundedness of solutions for $χ<\sqrt{\frac2N}$ in \cite{fujie} even relies on the second equation actually ensuring a positive pointwise lower bound for $v$. 

In \eqref{intro:system}, we cannot hope for such a convenient bound and thus have to deal with the influence of the actual singularity in the sensitivity function. \\

Nevertheless, for $D\equiv 1$, in the domains $ℝ^2$ and $ℝ^3$  a global existence result was achieved for initial data that are $H^1\times H^1$-close to $(\ubar,0)$ for some $\ubar>0$  (\cite{wang_xiang_yu}). 
The proof rests on energy estimates for a hyperbolic system into which \eqref{intro:system} can be converted by means of the Hopf-Cole type transformation $q:=\frac{\na v}v$ that had been introduced in \cite{levine_sleeman} for the treatment of an angiogenesis model.  


More recently it has become possible to treat general initial data (the only restrictions being positivity and regularity assumptions) for the system in bounded planar domains (\cite{win_ctsingabs}), where it was shown that global generalized solutions to \eqref{intro:system} with $D\equiv1$ exist  
whose second component $v$ moreover converges to $0$ with respect to the norm in any $L^p(\Om)$ for $p\in[1,∞)$ and to the weak-$*$ topology of $L^\infty(\Om)$. 
If, moreover, the initial mass of bacteria is small, the solution becomes eventually smooth (\cite{win_ctsingabseventual}) and converges to the homogeneous steady state. In \cite{win_ctsingabseventual} also an explicit smallness condition on $u_0$ in $L\log L (\Om)$ and $\na \ln v_0$ in $L^2(\Om)$ has been found that ensures the global existence of classical solutions. 

Solutions emanating from large data, however, have not been proven to be bounded and might blow up 
and cease to exist as classical solutions after a finite time, continuing only as generalized solutions in the sense of \cite{win_ctsingabs}. 
In higher-dimensional domains, even the existence of such solutions is unknown. Only in a radially symmetric setting ``renormalized solutions'' have been constructed (\cite{win_ctsingabsrenormalized}).\\

In the present article, we aim to find solutions to \eqref{intro:system} that are locally bounded and hence do not blow up in finite time. For this, we will rely on stronger growth of $D$, i.e. on the nonlinear diffusion we want to include. 
More precisely, we assume that with some $m\geq 1$, which will be subject to further conditions, and $δ>0$ 
\[
 D\in \Cdm:=\set{d\in C^1([0,∞)); d(s)\geq δs^{m-1} \text{ for all } s\in [0,∞)}. 
\]
In a first step we will additionally require strict positivity of $D$, i.e. 
\[
 D\in \Cdm^+:=\set{d\in C^1([0,∞)); d(s)\geq δs^{m-1} \text{ for all } s\in [0,∞) \text{ and } d(0)>0}
\]
and prove global existence of classical solutions to \eqref{intro:system}: 
\begin{theorem}\label{thm:nondegenerate}
 Let $N\geq 2$ and $\Om\subℝ^N$ be a bounded smooth domain. Then for every $δ>0$ and $m\geq 1$ satisfying 
\begin{equation}\label{mcond}
 m > 1+ \frac{N}4,
\end{equation}
every $D\in \Cdm^+$ and every pair $(u_0,v_0)$ of initial data fulfilling 
\begin{equation}\label{initcond}
 u_0\in C^{α}(\Ombar) \text{ for some }α\in(0,1), \qquad 
 v_0\in W^{1,\infty}(\Om), \qquad
 u_0\geq 0,\quad 
 v_0>0 \quad \text{ in }\Ombar
\end{equation}
the initial boundary value problem 
\begin{subequations}\label{sys}
\begin{align}
 u_t&=\na\cdot(D(u)\na u) - \na\cdot \kl{\frac uv\na v}&&\text{in } \Om\times(0,\Tmax)\label{sys:u}\\
 v_t&=\Delta v-uv&&\text{in } \Om\times(0,\Tmax)\label{sys:v}\\
 \delny u&=0&&\text{in } \dOm\times(0,\Tmax)\label{sys:ubdry}\\
 \delny v&=0&&\text{in } \dOm\times(0,\Tmax)\label{sys:vbdry}\\
 u(\cdot,0)&=u_0&&\text{in } \Om\label{sys:uinit}\\
 v(\cdot,0)&=v_0&&\text{in } \Om\label{sys:vinit}
\end{align}
\end{subequations}
has a classical solution $(u,v)\in (\usualspace)^2$ which is global (i.e. $\Tmax=\infty$). 
\end{theorem}
Afterwards dropping the strict positivity asumption on $D$, we will use an approximation procedure and finally prove the existence of global weak solutions that are locally bounded:
\begin{theorem}\label{thm:degenerate}
 Let $N\geq 2$ and $\Om\subℝ^N$ be a bounded smooth domain. Then for every $δ>0$ and $m>1+\frac N4$, every initial data 
\begin{equation}\label{initcond.weak}
 u_0\in\Lmme, \qquad v_0\in W^{1,\infty}(\Om), \qquad u_0\geq 0, \qquad v_0>0
\end{equation}
and every $D\in \Cdm$, \eqref{sys} has a global locally bounded weak solution $(u,v)$ (in the sense of Definition \ref{def:weaksol}), which in particular satisfies
\[
\norm[L^{\infty}(\Om\times(0,T))]{u}<∞ \qquad \text{for every } T\in(0,∞).
\]
\end{theorem}

We will devote Section \ref{sec:locex} to the proof of local existence of solutions and an extensibility criterion. 
In the proof of boundedness that follows, we will sometimes use the system 
\begin{equation}\label{uweq}
 \begin{cases}
 u_t=\na\cdot(D(u)\na u)-\na\cdot(u\na w)\\
 w_t=\Delta w -|\na w|^2 +u  
 \end{cases}
\end{equation}
obtained from the transformation $w=-\log(\frac{v}{\norm[\Liom]{v_0}})$, which has also been used in \cite{win_ctsingabs,win_ctsingabseventual}.
We note that while the first equation seems more accessible in \eqref{uweq} due to the lack of any singularity, it is \eqref{sys}, where the second equation is more amenable to the derivation of estimates on $\na v$.

The first stepping stone for the proof will be a spatio-temporal $L^2$-bound for $\na w$ (Lemma \ref{lem:intntionaw:2}), already giving some boundedness information for $\int_0^t\io |\na u^{m-1}|$ and $\io u^{m-1}(\cdot,t)$ for $t>0$, which we can use to obtain bounds on $\intnt\io |\na u^{m-1}|$ (Lemma \ref{lem:ionaumme}) and thereby on \(
 \intnt\norm[p]{u}^r 
\) 
for certain $p$, $r$ and $t>0$ (Lemma \ref{lem:ulpr}).
One consequence of such bounds is a spatio-temporal $L^q$-bound on $\na v$ (see Lemma \ref{lem:upr.gives.navq}), derived with the help of maximal Sobolev regularity properties of the heat equation (cf. Lemma \ref{lem:maxSobolev}). Another is the (local-in-time) boundedness of $w$ (Lemma \ref{lem:wbd}). This is important, as it will enable us to transfer bounds from $\na v$ to $\na w$ (Lemma \ref{lem:nav.gives.naw}).

Bounds on $\intnt\io |\na w|^q$ now in turn will translate into control over $\io u^p$ for some $p$ (Lemma \ref{lem:nawq.gives.up.for.pgeqn}). If $p$ is sufficiently large, this entails $L^\infty(\Om\times(0,T))$- boundedness of $|\na v|$ and $|\na w|$ and thus finally of $u$ (Lemma \ref{lem:naw.q.gives.u.p} and Lemma \ref{lem:ex.estimates.parab} \ref{lem:exestimparab:taowin}). Thereby, the solution is not only locally bounded, but moreover exists globally, according to the extensibility criterion \eqref{eq:extcrit}.

In Section \ref{sec:weaksol} we rely on bounds already derived in the previous section to construct locally bounded weak solutions to \eqref{intro:system} with functions $D$ causing possibly degenerate diffusion.


\textbf{Notation.} Throughout the article we fix $N\in ℕ$, $N\geq 2$, and $\Om\sub ℝ^N$ as a bounded, smooth domain. When dealing with the solution to a differential equation, we will use $\Tmax$ to denote its maximal time of existence; in the case of \eqref{sys} such $\Tmax$ is provided by Lemma \ref{lem:locex:and:continuation}. By $\embeddedinto$ and $\cptembeddedinto$ we refer to continuous and compact embeddings of Banach spaces, respectively. We will sometimes write $D(u)$ for the concatenation $D\circ u$ of functions. The number $\lambda_1>0$ will always be the first positive eigenvalue of the Neumann Laplacian.

\section{Local existence}
\label{sec:locex}
We begin the proof by ensuring local existence of classical solutions in the non-degenerate case. As a first step let us, for easier reference, collect some basic results on existence of and estimates for solutions of certain parabolic PDEs.

\begin{lemma}\label{lem:ex.estimates.parab}
\beenum
  \item \label{lem:exestimparab:vholder}
 For any $T>0$, $q>N$ and $r>N$ and every $M>0$ there are $C_{i}>0$ and $γ>0$ such that for all nonnegative functions $v_0\in\Wom1q$ and $u\in \LT{∞}{\Lom r}$ satisfying $\norm[\Wom1q]{v_0}\leq M$ and $\norm[\LT{∞}{\Lom r}]{u}\leq M$ for the solution $v\in V_2=\{v\in L^\infty((0,T);\Lom2); \na v\in L^2(\Om\times(0,T))\}$ of 
\begin{equation}\label{eq:estlemma:v}
 v_t=Δv - uv, \quad \delny v\amrand=0, \quad v(\cdot,0)=v_0
\end{equation}
 one has $\normm{C^{γ,\frac{γ}2}(\Ombar\times[0,T])}{v}<C_{i}$. 
 If, moreover, $u\in C^{α,\frac{α}2}(\Ombar\times(0,T])$ for some $\alpha\in(0,γ)$, then $v\in C^0(\Ombar\times[0,T])\cap C^{2+α,1+\frac{α}2}(\Ombar\times(0,T])$. If $u\in L^\infty(\Om\times(0,T))$, then $v\in C^1(\Ombar\times(0,T])$.
\item \label{lem:exestimparab:nav}
  For any $r∈(N,∞]$ there is $C_i=C_{ii}(r)>0$ such that for any $T>0$ and any $q\in[2,∞]$ for all nonnegative functions $v_0\in \Wom1q$ and $u\in C^{α,\frac{α}2}(\Ombar\times(0,T))$ for some $α\in(0,1)$, the solution $v\in C^{γ,\frac{γ}2}(\Ombar\times[0,T])$ (for some $γ\in(0,α)$) of 
\eqref{eq:estlemma:v} satisfies 
 \[
  \norm[\Lom q]{∇v(\cdot,t)}\leq C_i\norm[\Lom q]{∇v_0} + C_{ii}\norm[\Lom{∞}]{v_0}\norm[L^{∞}((0,T);\Lom r)]{u}.
 \]
 \item \label{lem:exestimparab:ulinfty}
 For every $T>0$, $δ_0>0$, $M>0$ and $K>0$ there is $C_{iii}>0$ such that for every $u_0∈\Lom{∞}$ satisfying $0\leq u_0\leq M$ and every $g\in\kl{C^0((\Ombar\times(0,T))}^N$ fulfilling $g\cdotν=0$ on $\dOm$ and $\norm[\Lom{∞}]{g}\leq K$, and for all $A\in\Li(Ω\times(0,T))$ with $A>δ_0$ in $Ω\times(0,T)$, the unique weak solution of 
\begin{equation}\label{eq:estlemma:u}
 u_t=∇\cdot (A∇u-g) \text{ in } \Om\times(0,T), \quad \delny u\amrand=0 \text{ in } (0,T), \quad u(\cdot,0)=u_0 \text{ in } \Om, 
\end{equation}
 satisfies 
 \begin{equation}\label{eq:uinfty:regestimate}
  \norm[\Li(Ω\times(0,T))]{u}\leq C_{iii} \quad\text{and}\quad \norm[L^2(Ω\times(0,T))]{∇u}\leq C_{iii}.
 \end{equation}
 \item \label{lem:exestimparab:uholder}
  For any $T>0$, for any $D_0>δ_0>0$, $M>0$, $K>0$ and $α∈(0,1)$ there are $C_{iv}>0$ and $γ∈(0,1)$ such that for every $A\in \Li(\Omega\times(0,T))$ fulfilling $δ_0<A<D_0$ a.e. in $\Om\times(0,T)$, and for all $g∈\kl{C^0(\Ombar\times(0,T))}^N$ with $g\cdotν=0$ on $\dOm$ and $\norm[\Li(Ω\times(0,T))]{g}\leq M$ and all $u_0\in C^{α}(\Ombar)$ with $\norm[C^{α}(\Ombar)]{u_0}\leq M$, any solution $u$ of \eqref{eq:estlemma:u} that obeys the estimate $\norm[\Li(Ω\times(0,T))]{u}\leq K$ satisfies 
\begin{equation}\label{reg:uholder:estimate}
 \normm{C^{γ,\frac{γ}2}(\Ombar\times[0,T])}{u}\leq C_{iv}.
\end{equation}
 Moreover, if $g∈C^{β,\frac{β}2}(\Ombar\times(0,T])$ for some $β>0$, then $u\in C^{2,1}(\Ombar\times(0,T])$.
\item \label{lem:exestimparab:taowin}
For every $m\geq 1$, $δ>0$, $K>0$, $p_0\geq 1$, $q_1>n+2$ and $T\in(0,∞]$ there is $C_{v}>0$ such that for every $D\in C^1(\Ombar\times[0,T)\times[0,\infty))$ which obeys $D\geq 0$, $D(x,t,s)\geq δs^{m-1}$ for all $(x,t,s)\in\Om\times(0,T)\times(0,∞)$ and every $f\in C^0((0,T);C^0(\Ombar)\cap C^1(\Om)), f\cdotν\leq 0$ on $\dOm\times(0,T)$ satisfying $\norm[L^\infty((0,T);\Lom{q_1})]{f}\leq K$, for every nonnegative function $u\in C^0(\Ombar\times[0,T))\cap C^{2,1}(\Ombar\times(0,T))$ that satisfies $\norm[L^\infty((0,T);\Lom{p_0})]{u}\leq K$ and $u_t\leq \nabla\cdot(D(x,t,u)\nabla u)+∇\cdot f(x,t)$ in $\Om\times(0,T)$ and $\delny u\amrand\leq 0$ on $(0,T)$, we have $\norm[\Liom]{u(\cdot,t)}\leq C_v$ for every $t\in(0,T)$.
\end{enumerate}

\end{lemma}
\begin{proof}
\textbf{\ref{lem:exestimparab:vholder}} 
According to \cite[III.5.1]{LSU}, \eqref{eq:estlemma:v} has a unique weak solution $v\in V_2$ in $Ω\times(0,T)$. 
The first part of the statement thus immediately results from \cite[Thm. 1.3 and Remarks 1.3, 1.4]{porzio_vespri}, whereas the second is a consequence of a uniqueness statement (\cite[III.5.1]{LSU}) combined with the existence assertion for classical solutions in \cite[IV.5.3]{LSU} (applied to $ζ_{ε}(t)v(x,t)$ for some cutoff function $ζ_{ε}\in C_0^\infty([0,∞))$, $ζ_{ε}\restr{(0,\frac{ε}2)}\equiv 0$, $ζ_{ε}\restr{(ε,∞)}\equiv 1$ for arbitrary $ε>0$). The third part - actually, even Hölder-continuity of $\na v$ - is provided by \cite[Thm. 1.1]{lieberman_holdergradient}.

\textbf{\ref{lem:exestimparab:nav}} Existence of a solution ensured as in the proof of \ref{lem:exestimparab:vholder}, we may rely on \cite[Cor. 4.3.3]{pazy} to represent $v$ as mild solution via the variation of constants formula, and invoking \cite[Lemma 1.3 iii)]{win_aggregationvs} and \cite[Lemma 1.3 ii)]{win_aggregationvs}, we gain $c_1>0$ and $c_2>0$, respectively, such that with $ρ:=\max\set{q,r}$ and by Hölder's inequality
 \begin{align*}
   \norm[\Lom q]{∇v(\cdot,t)}&\leq\! c_1 \norm[\Lom q]{∇v_0}\! +\! \intnt c_2 |Ω|^{\frac1q-\frac1{ρ}}\kl{1+(t-s)^{-\frac12-\frac N2(\frac1r-\frac1{ρ})}}\! \norm[\Lom r]{uv(\cdot,s)}\!e^{-λ_1(t-s)}\!\ds\\
   &\leq c_1\norm[\Lom q]{∇v_0}+c_3\norm[\Lom{∞}]{v_0}\norm[\LT{∞}{\Lom r}]{u},
 \end{align*}
 where we have used that $σ^{-\frac12-\frac N2(\frac1r-\frac1{ρ})}\leq 1+σ^{-\frac12-\frac N{2r}}$ for all $σ>0$ and all $ρ∈[1,∞]$ and set $c_3:=\intninf c_2\kl{2+σ^{-\frac12-\frac N{2r}}}e^{-λ_1σ}\dsigma$, which is finite because of $r>N$, and where we have taken into account that by comparison arguments 
 $0\leq v(\cdot,t)\leq v_0$ in $Ω$ for all $t∈(0,T)$.

\textbf{\ref{lem:exestimparab:ulinfty}} This is a combination of \cite[Thm. 6.38]{lieberman_book} (estimate for $\na u$), \cite[Thm. 6.39]{lieberman_book} (existence and uniqueness) and \cite[Thm. 6.40]{lieberman_book} (uniform boundedness).

\textbf{\ref{lem:exestimparab:uholder}} The same theorems as in the proof of \ref{lem:exestimparab:vholder} apply. 
 
\textbf{\ref{lem:exestimparab:taowin}} This is (part of) Lemma A.1 in \cite{taowin_bdnessquasilinsubcritical}.
\end{proof}

\begin{lemma}\label{lem:locex_first}
For every positive function $D\in C^0([0,∞))$, for every $L>0$ and $α\in(0,1)$ there is $T>0$ such that 
for every $u_0\in C^{\alpha}(\Ombar)$ satisfying $\norm[C^{α}(\Ombar)]{u_0}\leq L$ and every $v_0\in \Wom1{∞}$ which satisfies $v_0>\frac1L$ in $\Ombar$ and fulfils $\norm[\Wom1{∞}]{v_0}\leq L$ there is a pair of functions $(u,v)\in C^0(\Ombar\times[0,T])∩C^{2,1}(\Ombar\times(0,T])$ solving \eqref{sys} in $Ω\times(0,T)$. 
\end{lemma}
\begin{proof}
 We let $R:=L+1\geq \norm[\Liom]{u_0}+1$.
 For the choice of $r:=∞$ we obtain $c_1:=C_{ii}>0$ with properties as described in Lemma \ref{lem:ex.estimates.parab} \ref{lem:exestimparab:nav}, and thereupon invoking Lemma \ref{lem:ex.estimates.parab} \ref{lem:exestimparab:ulinfty} for parameters $T=1$, $δ_0=\inf_{s>0} D(s)$, $M=R$, $K=c_1 R (1+R) L^2 e^R$, we are given  $c_2:=C_{iii}>0$ as in \eqref{eq:uinfty:regestimate}. 
 An application of Lemma \ref{lem:ex.estimates.parab} \ref{lem:exestimparab:uholder} for $T=1$, $δ_0=\inf_{s>0} D(s)$, $D_0=\sup_{0<s<R} D(s)$, $M=\max\set{c_1R(1+R)L^2 e^R, L}$ and $K=c_2$ provides us with $c_3:=C_{iv}>0$ and $\gamma\in(0,1)$ as in \eqref{reg:uholder:estimate}. 
 With these, we choose $T∈(0,1)$ such that $\norm[\Liom]{u_0}+c_3T^{\frac\gamma2}<R$ and introduce 
 \begin{equation}\label{def:S}
  S:=\set{\uhat\in C^0(\Ombar\times[0,T]); 0\leq \uhat\leq R, u(\cdot,0)=u_0, \normm{C^{γ,\frac{γ}2}(\Ombar\times[0,T])}{\uhat}\leq c_3}\subset C^0(\Ombar\times[0,T]).
 \end{equation}

 For any $\uhat\in S$ we define $\uhat(t):=u(T)$ for $t\in(T,1]$ and note that the solution $v$ of 
\begin{equation}\label{eq:locex:v}
 v_t=Δv-\uhat v \quad \text{in } Ω\times(0,1), \quad \delny v\amrand=0,\quad v(\cdot ,0)=v_0 \text{ in } \Om,
\end{equation}
satisfies 
\begin{equation}\label{eq:lowerbd_v}
 \norm[\Liom]{v_0}\geq v(\cdot ,t)\geq \kl{\inf v_0}e^{-R}\geq \frac1 L e^{-R}
\end{equation}
in $Ω$ for all $t\in [0,1]$ and, by definition of $c_1$, $\norm[\Liom]{∇v(\cdot ,t)}\leq c_1(1+R)\norm[\Wom1{∞}]{v_0}\leq c_1(1+R)L$.

We let $u$ be the solution of 
\[
 u_t=\nabla\cdot\kl{D(\uhat)\na u - \frac{\uhat}v\na v} \text{ in } \Om\times(0,1), \quad \delny u\amrand=0, \quad u(\cdot,0)=u_0 \text{ in } \Om.
\]

Then by definition of $c_2$ and $c_3$ (with $g=\frac{\uhat}v∇v$ 
and $A=D\circ\uhat$ in \eqref{eq:estlemma:u}), $\norm[\Li(Ω\times(0,1))]{u}\leq c_2$ and $\normm{C^{γ,\frac{γ}2}(\Ombar\times[0,1])}{u}\leq c_3$. 
Hence if we define $Φ(\uhat):=u\restr{Ω\times(0,T)}$, we have $\norm[\Li(Ω)]{Φ(\uhat)(t)}\leq \norm[\Liom]{u_0}+c_3t^{\frac{γ}2}\leq R$ for every $t\in(0,T)$ and every $\uhat\in S$, and thus $Φ$ is a function mapping $S$ into itself, where $S$ is a closed convex set in $C^0(\Ombar\times[0,T])$. Moreover, $Φ\colon S\to S$ is continuous: 
We let $\ubar\in S$ and $\uhat^k\in S$ for all $k\in ℕ$ such that $\uhat^k\to \ubar$ in $C^0(\Ombar\times[0,T])$. Then, with respect to $\norm[\Li(\Om\times(0,T))]{\cdot}$ and with respect to the weak-*-topology of $\LT{∞}{\Wom1{∞}}$, the solutions $v^k$ of \eqref{eq:locex:v} with $\uhat$ replaced by $\uhat^k$ converge to $\vbar$ solving \eqref{eq:locex:v} with $\ubar$ instead of $\uhat$: 
Assuming on the contrary that there were a sequence $(k_l)_{l\inℕ}$ such that for each subsequence $(k_{l_m})_{m\inℕ}$ therof the sequence $(v^{k_{l_m}})_{m\inℕ}$ did not converge in the indicated topologies, from the uniform bounds on 
$\normm{C^{γ,\frac{γ}2}(\Ombar\times[0,T])}{v^{k_l}}$ and on $\norm[\LT{∞}{\Liom}]{∇v^{k_l}}$ asserted by Lemma \ref{lem:ex.estimates.parab} \ref{lem:exestimparab:vholder} and Lemma \ref{lem:ex.estimates.parab} \ref{lem:exestimparab:nav}, respectively, we could conclude the existence of some subsequence $(v^{k_{l_n}})_{n\in ℕ}$ being uniformly convergent in $Ω\times[0,T]$ and weakly-$^\ast$-convergent in $\LT{∞}{\Wom1{∞}}$. By passing to the limit in the weak formulation in the equations of the form \eqref{eq:locex:v} satisfied by $v^{k_l}$, the limit can easily be seen to coincide with the unique weak solution $\vbar$ of \eqref{eq:locex:v} with $\ubar$ replacing $\uhat$, contradicting the choice of $(v^{k_l})_{l\inℕ}$. We observe that hence and by \eqref{eq:lowerbd_v} $\frac{u^k}{v^k}∇v^k\weakstarto \frac{\ubar}{\vbar}∇\vbar$ in $\LT{∞}{\Lom{∞}}$. 
Similarly taking into account bounds on $\normm{C^{γ,\frac{γ}2}(\Ombar\times[0,T])}{Φ(\uhat^k)}$ and $\norm[L^2(Ω\times(0,T))]{∇Φ(\uhat^k)}$ as obtained from Lemma \ref{lem:ex.estimates.parab} \ref{lem:exestimparab:uholder} and \ref{lem:ex.estimates.parab} \ref{lem:exestimparab:ulinfty} and again employing the weak formulation of the equations defining $Φ(\uhat^k)$ and uniqueness of the solution $Φ(\ubar)=u$ of $u_t=∇\cdot (D(\ubar)∇u-\frac{\ubar}{\vbar}∇\vbar)$, $\delny u\amrand=0$, $u(\cdot ,0)=u_0$, we finally see that $Φ(\uhat^k)→Φ(\ubar)$ in $C^0(\Ombar\times[0,T])$.

We note that $S⊂C^0(\Ombar\times[0,T])$ is a closed bounded convex set and $Φ(S)$ is relatively compact in $C^0(\Ombar\times[0,T])$, owing to the uniform Hölder bound $c_3$ and Arzel\`a-Ascoli's theorem, so that we can apply Schauder's fixed point theorem to find $u\in S$ such that $Φ(u)=u$. 
Due to the regularity assertions in Lemma \ref{lem:ex.estimates.parab} \ref{lem:exestimparab:uholder} even $u\in C^{2,1}(\Ombar\times(0,T])$; also the corresponding solution $v$ of the second equation belongs to this space by \ref{lem:ex.estimates.parab} \ref{lem:exestimparab:vholder}. 
\end{proof}

\begin{lemma}\label{lem:cancontinue}
Let $T>0$. 
 If on $\Omega\times(0,T)$ there is a solution $(u,v)$ to \eqref{sys} such that 
\[
 \sup_{t\in(0,T)} \norm[\Lom {∞}]{u(\cdot ,t)}<\infty, 
\]
then there is $\Ttilde>T$ such that there is a solution to \eqref{sys} in $\Omega\times(0,\Ttilde)$ which on $\Omega\times(0,T)$ coincides with $(u,v)$.
\end{lemma}
\begin{proof}
Successive application of comparison arguments in \eqref{sys:v} and of Lemma \ref{lem:ex.estimates.parab} parts \ref{lem:exestimparab:nav},  \ref{lem:exestimparab:uholder} and \ref{lem:exestimparab:vholder} show the existence of $α>0$ and $M>0$ such that 
\begin{align*}
 &\inf_{\Omega\times(0,T)} v>\frac1M, & &\norm[\LT{∞}{\Wom1{∞}}]{v}\leq M,&
 &\normm{C^{α,\frac{α}2}(\Ombar\times[0,T])}{u}\leq M,  & &\normm{C^{α,\frac{α}2}(\Ombar\times[0,T])}{v}\leq M.
\end{align*}
 Due to the uniform continuity of $u$ and $v$,
\[
 \utilde_0(x):=\lim_{t\upto T} u(x,t), \qquad \vtilde_0(x):=\lim_{t\upto T}v(x,t), \qquad x\in \Ombar, 
\]
are well-defined and satisfy 
\(
 M\geq \vtilde_0 \geq \frac1M$ in  $\Ombar$ as well as $\norm[C^{α}(\Ombar)]{\utilde_0}\leq M$. 

Picking a sequence $(t_k)_{k∈ℕ}\upto T$ and referring to $\norm[\LT{∞}{\Wom1{∞}}]{v}\leq M$, we may conclude the existence of a subsequence $(t_{k_l})_{l∈ℕ}$ such that $∇v(\cdot ,t_{k_l})\weakstarto ∇\vtilde_0$ in $\Liom$ as $l→∞$, and thus infer $\norm[\Wom1{∞}]{\vtilde_0}\leq M$. According to Lemma \ref{lem:locex_first}, we can find $τ>0$ and $(\utilde,\vtilde)\in \kl{C^0(\Ombar\times[0,τ])\cap C^{2,1}(\Ombar\times(0,τ])}^2$ solving 
\begin{align*}
 \utilde_t&=\na \cdot\kl{D(\utilde)\na\utilde - \frac{\utilde}{\vtilde}\na \vtilde},\qquad  \vtilde_t= \Delta \vtilde - \utilde\vtilde\qquad \text{ in } \Omega\times(0,τ),\\ &\delny \utilde\amrand=\delny\vtilde\amrand=0,\;\; \utilde(\cdot,0)=\utilde_0,\; \vtilde(\cdot,0)=\vtilde_0.  
\end{align*}
Letting 
\[
 (\ubar,\vbar)(\cdot,t):=\begin{cases}(u,v)(\cdot,t),&t<T,\\(\utilde,\vtilde)(\cdot,t-T),&t\in[T,T+τ),\end{cases}
\]
we obtain a weak solution of \eqref{sys} in $\Om\times(0,T+τ)$, which by Lemma \ref{lem:ex.estimates.parab} parts \ref{lem:exestimparab:uholder} and \ref{lem:exestimparab:vholder} is classical. 
\end{proof}

\begin{lemma}\label{lem:locex:and:continuation}
 Let $\alpha\in(0,1)$, $m\geq 1$, $δ>0$. 
 For every $D\in\Cdm$, $u_0\in C^{α}(\Ombar)$, $v_0 \in \Wom1{∞}$, $u_0\geq 0$, $v_0>0$ in $\Ombar$,  there is $\Tmax>0$ and $(u,v)\in(\usualspace)^2$ such that $(u,v)$ solves \eqref{sys} and 
\begin{equation}\label{eq:extcrit}
 \Tmax=∞ \quad \text{or }\quad \limsup_{t\upto\Tmax} \norm[\Lom \infty]{u(\cdot,t)}=\infty.
\end{equation}
Moreover, $u\geq 0$ and $0<v<\norm[\Lom{∞}]{v_0}$ throughout $\Om\times(0,\Tmax)$.
\end{lemma}
\begin{proof}
We let $u_0\in C^{α}(\Ombar)$ and $v_0\in \Wom1{∞}$, define 
\[
 \calS=\set{(t,u,v);\, t\in(0,∞), u,v\in \usualspace \;\; (u,v) \text{ solves } \eqref{sys}}
\]
and introduce the order relation $\preceq$ given by 
\[
 (t_1,u_1,v_1)\preceq(t_2,u_2,v_2) :\iff t_1\leq t_2, u_2|_{(0,t_1)}=u_1, \, v_2|_{(0,t_1)}=v_1. 
\]
Every totally ordered set $M_I=\set{(t_i,u_i,v_i);\; i\in I}$ with arbitrary index set $I$ has an upper bound $(\sup_{i\in I} t_i,u,v)$, where $u(τ)=u_i(τ)$ if $τ\in(0,t_i)$ and $v$ is defined analogously. (This yields well-defined functions, since $u_{i_1}(τ)=u_{i_2}(τ)$ if $τ\in(0,t_{i_1})\cap (0,t_{i_2})$, because $M_I$ is totally ordered.) Moreover, $\calS$ is not empty, according to Lemma \ref{lem:locex_first}. 
By Zorn's lemma there is some maximal element $(\Tmax,u,v)\in \calS$. Assume that \mbox{$\limsup_{t\upto \Tmax} \norm[\Lom \infty]{u(\cdot,t)}<\infty$}. Then Lemma \ref{lem:cancontinue} immediately yields $\Ttilde>T$ such that $(\Ttilde,\utilde,\vtilde)\succeq (\Tmax,u,v)$, contradicting the maximality of $(\Tmax,u,v)$.
\end{proof}

\section{The nondegenerate case. Proof of Theorem \ref{thm:nondegenerate}}\label{sec:nondeg}
This section is devoted to the derivation of estimates for the solutions, so as to finally obtain their global existence by means of the extensibility criterion \eqref{eq:extcrit}. 

For some manipulations in \eqref{sys:u} it would be more convenient to deal with a nonsingular chemotaxis term of the form $∇\cdot (u∇w)$ instead of $∇\cdot (\frac{u}v∇v)$. For this purpose, we employ the following transformation and, given a solution $(u,v)\in(\usualspace)^2$ of \eqref{sys} for initial data $(u_0,v_0)$ as in \eqref{initcond}, let 
\begin{equation}\label{eq:def:w}
 w:=-\log \frac{v}{\norm[\Liom]{v_0}} \qquad \text{in } \Omega\times[0,\Tmax). 
\end{equation}
Then $w\geq 0$ in $\Omega\times(0,\Tmax)$ and $(u,w)\in(\usualspace)^2$ solves 
\begin{subequations}\label{eq:uw-system}
\begin{align}
  u_t&=∇\cdot (D(u)∇u+u∇w)&&\text{in } \Om\times(0,\Tmax)\label{uw:u}\\
  w_t&=Δw-|∇w|^2+u&&\text{in } \Om\times(0,\Tmax)\label{uw:w}\\
  &\delny u=0=\delny w&&\text{in } \dOm\times(0,\Tmax)\label{uw:rand}\\
  &u(\cdot,0)=u_0, \quad w(\cdot,0)=w_0:= -\log \frac{v_0}{\norm[\Liom]{v_0}}&&\text{in } \Om,\label{uw:init} 
\end{align}
\end{subequations}
where $(u_0,w_0)$ satisfy 
\begin{equation}\label{uw-init}
 u_0\geq 0,\quad w_0\geq 0,\quad w_0\in W^{1,\infty}(\Om), \quad u_0\in C^{α}(\Ombar) \text{ for some } α\in(0,1).
\end{equation}

Evidently, given $\norm[\Liom]{v_0}$ every solution $(u,w)$ to \eqref{eq:uw-system} yields a solution to \eqref{sys} via $v:=\norm[\Liom]{v_0}e^{-w}$. 

We will proceed in several steps, the first being the following simple observation that the bacterial mass is conserved throughout evolution: 
\begin{lemma}\label{lem:massconservation:u}
Let $T>0$, $\delta>0$, $m\inℝ$, $D\in \Cdm$ and let $u\in C^0(\Ombar\times[0,T))\cap C^{2,1}(\Ombar\times(0,T))$ solve  \eqref{sys:ubdry}, \eqref{sys:u} with some $v\in C^{2,1}(\Ombar\times(0,T))$ or \eqref{uw:u}, \eqref{uw:rand} with some $w\in C^{2,1}(\Ombar\times(0,T))$. Then, with $u_0:=u(\cdot,0)$,    
 \[
  \io u(\cdot ,t)=\io u_0\qquad \text{for every } t\in (0,T).
 \]
\end{lemma}
\begin{proof} 
This is an immediate consequence of \eqref{sys:u} or \eqref{uw:u}, obtained upon integration over $\Om$.  
\end{proof}

In \eqref{eq:uw-system}, a spatio-temporal $L^2$-bound for $∇w$ can be inferred rather directly:
\begin{lemma}\label{lem:intntionaw:2}
Let $m\inℝ$, $δ>0$, $T>0$ and $(u,w)\in(\usualspace[T])^2$ be a solution to \eqref{eq:uw-system} for any $D\in\Cdm$. Then, with $u_0:=u(\cdot,0)$, $w_0:=w(\cdot,0)$,  
 \[
  \intnt\io|∇w|^2\leq \io w_0+t\io u_0\qquad \text{for every } t\in(0,T).
 \]
\end{lemma}
\begin{proof}
 In order to see this, it is sufficient to integrate the second equation of \eqref{eq:uw-system} and take into account Lemma \ref{lem:massconservation:u}.
\end{proof}

This bound can be transformed into a first information on derivatives of $u$: 
\begin{lemma}\label{lem:ionaumme}
 Let $m>1$, $δ>0$. For any $K>0$ there is $C>0$ such that for any $D\in \Cdm$ any solution $(u,w)$ of \eqref{eq:uw-system} emanating from initial data $(u_0,w_0)$ as in \eqref{uw-init} with $\norm[\Lmme]{u_0}\leq K$, $\norm[\Lom1]{w_0}\leq K$ obeys the estimates 
\[
 \intnt\io u^{2m-4}|∇u|^2 \leq C(1+t) \fatntmax 
\]
and 
\[
 \intnt\io D(u)u^{m-3}|\na u|^2\leq C(1+t) \fatntmax  
\]
as well as 
\label{lem:intumme}
\begin{equation}\label{eq:intumme}
 \io u^{m-1}(\cdot,t)\leq C(1+t) \fatntmax.
\end{equation}
\end{lemma}
\begin{proof}
Due to \eqref{uw:u}, on $(0,\Tmax)$ we have 
 \begin{align}\label{eq:ddtumme}
  \ddt \io u^{m-1} &= (m-1)\io u^{m-2}∇\cdot (D(u)∇u+u∇w)\nn\\
&=(m-1)(2-m)\io D(u)u^{m-3}|∇u|^2+(m-1)(2-m)\io u^{m-2}∇u\cdot∇w.
 \end{align}
 We note that by Young's inequality 
\begin{equation}\label{eq:lem.ionaumme:young}
 \left\lvert\io u^{m-2}∇u\cdot ∇w\right\rvert \leq \frac{δ}3\io u^{2m-4}|∇u|^2+\frac3{4δ}\io |∇w|^2\quad \text{ on } (0,\Tmax).
\end{equation}
The sign of $(m-1)(2-m)$ in \eqref{eq:ddtumme} depends on the size of $m$ and we therefore distinguish the following cases:

If $m\in(1,2)$, \eqref{eq:ddtumme} together with \eqref{eq:lem.ionaumme:young} and Lemma \ref{lem:intntionaw:2} 
yields 
\begin{align}\label{eq:intntionaumme:msmall}
 \frac13&\intnt\io D(u)u^{m-3}|\na u|^2 + \frac{δ}3\intnt\io u^{2m-4}|∇u|^2 \nn\\
&\leq\frac{1}{(m-1)(2-m)}\io u^{m-1}(\cdot,t)-\frac1{(m-1)(2-m)} \io u_0+\frac3{4δ} \intnt\io |∇w|^2\nn\\
 &\leq \frac1{(m-1)(2-m)}|\Om|^{\frac{m-2}{m-1}} \kl{\io u_0}^{\frac1{m-1}}+\frac3{4δ}\kl{\io w_0+t\io u_0} \text{  for any } t\in(0,\Tmax).
\end{align}
If $m>2$, \eqref{eq:ddtumme} and \eqref{eq:lem.ionaumme:young} can be combined to give 
\begin{align*}
 \frac{1}{(m-1)(m-2)}&\io u^{m-1}(\cdot ,t)-\frac1{(m-1)(m-2)}\io u_0^{m-1} +\frac13 \intnt\io D(u)u^{m-3}|\na u|^2\\&+ \frac{δ}3 \intnt\io u^{2m-4}|∇u|^2
 \leq \frac3{4δ} \intnt\io |∇w|^2\quad \text{ for any } t\in(0,\Tmax),
\end{align*}
which allows for a similarly obvious definition of $C$ as \eqref{eq:intntionaumme:msmall}. 
This inequality also entails \eqref{eq:intumme} for $m>2$, the only case that does not immediately result from Lemma \ref{lem:massconservation:u}.

If $m=2$, apparently the consideration of $\ddt \io u^{m-1}=\ddt\io u=0$ does not help in achieving an estimate for $\intnt\io u^{2m-4}|\na u|^2=\intnt\io |\na u|^2$. From the analogously obtained 
\begin{align*}
 \ddt \io u\log u + \frac13 \io \frac{D(u)}u |\na u|^2 + \frac{2δ}3\io |∇u|^2 \leq \frac{δ}3\io |∇u|^2 + \frac3{4δ}\io |∇w|^2 \text{ on } (0,\Tmax),
\end{align*}
however, we obtain the same form of estimates as in the other cases.
\end{proof}

For convenience let us recall those special cases of the \GNI\ we are going to use in the following: 

\begin{lemma}\label{lem:GNI}
\beenum
\item\label{GNI:firstderivative} Let $0<q\leq p\leq \frac{2N}{N-2}$ (or $0<q\leq p<\infty$ if $N=2$) and let $s>0$ and $\gamma>0$. 
 Then there is $c>0$ such that 
\[
 \norm[\Lom p]{u}^{γ}\leq c\norm[\Lom 2]{\na u}^{aγ} \norm[\Lom q]{u}^{(1-a)γ} + c\norm[\Lom s]{u}^{γ} \qquad \text{for all } u\in W^{1,2}(\Om)\cap \Lom q\cap \Lom s,
\]
where 
\[
 a=\frac{\frac1q-\frac1p}{\frac1q+\frac1n-\frac12}. 
\]
\item\label{GNI:secondderivative}
Let $p,q\in(1,∞)$ be such that $p(q-N)=q(2p-N)a$ for some $a\in [\frac12,1)$. 
Then there is $c>0$ such that 
\[
 \norm[\Lom q]{\nabla v}^q \leq c\norm[\Lom p]{\Delta v}^{qa} \norm[\Liom]{v}^{q(1-a)}+c\norm[\Liom]{v}^q 
\]
for all $v\in \Wom2p\cap \Wom1q \cap \Liom$ with $\delny v=0$ on $\dOm$.
\end{enumerate}
\end{lemma}
\begin{proof}
 The \GNI\ can be found in \cite[p. 125]{nirenberg}, \cite[Thm. 10.1]{friedman} or in \cite[Lemma 2.3]{yan_joh} (where also the case of $p,q<1$ in \ref{GNI:firstderivative} is covered); replacing $D^2v$ by $\Delta v$ in the standard formulation of \ref{GNI:secondderivative} is possible by, e.g., \cite[Thm. 19.1]{friedman}. 
\end{proof}

Aided by the \GNI, in the next step, as consequence of the estimates from Lemma \ref{lem:ionaumme} we shall acquire the bound \eqref{eq:ulpr}, which will be featured as condition in Lemmas \ref{lem:upr.gives.navq} and \ref{lem:wbd.ifupr}, and can be seen as an important ingredient of the proof of Theorem \ref{thm:nondegenerate}.

\begin{lemma}\label{lem:ulpr}
Let $K>0$, $T\in(0,∞)$, $δ>0$ and $m>1$. If either
\beenum
 \item \label{ulpr;smallm}
$m\leq 2$, $r>1$, $p\geq 1$ satisfy $p\leq \frac{2N}{N-2}(m-1)$ and $r(1-\frac1p)\leq 2m-3+\frac2N$ or
\item \label{ulpr;largem}
$m\geq 2$, $r>1$ and $p\in[m-1,\frac{2N}{N-2}(m-1)]$ are such that $(\frac1{m-1}-\frac1p)r\leq 1+\frac2N$,
\end{enumerate}
then there is $C>0$ such that 
whenever $(u,w)\in(\usualspace[T])^2$ solves \eqref{eq:uw-system} for some $D\in\Cdm$ and for initial data $(u_0,w_0)$ with \eqref{uw-init} and $\norm[\Lmme]{u_0}\leq K, \norm[\Lom1]{w_0}\leq K$, we have 
\begin{equation}\label{eq:ulpr}
 \intnt\norm[\Lom p]{u}^r < C(1+t^{r+1}) \quad \text{for all } t\in(0,T). 
\end{equation}
\end{lemma}
\begin{proof}
\textbf{\ref{ulpr;smallm}} Due to $m>1$, the inequality $p\leq \frac{2N}{N-2}(m-1)$ is equivalent to $\frac{p}{m-1}\leq \frac{2N}{N-2}$, and $p\geq 1$ ensures $\frac{p}{m-1}\geq \frac1{m-1}$. Thus, the \GNI\ (Lemma \ref{lem:GNI} \ref{GNI:firstderivative}) yields $c_1>0$ such that with 
\[
 a:= \frac{m-1-\frac{m-1}p}{m-1+\frac1N-\frac12},
\]
 and hence $\frac r{m-1}a\leq 2$, 
 for all $t\in(0,\Tmax)$ we obtain
\begin{align*}
 \intnt\norm[\Lom p]{u}^r &=\intnt\norm[\Lom{\frac p{m-1}}]{u^{m-1}}^{\frac{r}{m-1}}\\
 &\leq c_1\intnt \norm[\Lom 2]{∇u^{m-1}}^{\frac r{m-1} a}\norm[\Lom{\frac1{m-1}}]{u^{m-1}}^{\frac{r}{m-1}(1-a)} + c_1 \intnt \norm[\Lom {\frac1{m-1}}]{u^{m-1}}^{\frac r {m-1}}\\
 &\leq c_1 \kl{\io u_0}^{\frac{r}{m-1}(1-a)} \intnt \kl{1+\norm[\Lom 2]{∇u^{m-1}}^2} + c_1\intnt \kl{\io u_0}^{\frac r{m-1}},
\end{align*}
where we have used Lemma \ref{lem:massconservation:u}, 
and can conclude the proof with applications of Lemma \ref{lem:ionaumme} and Young's inequality.

\textbf{\ref{ulpr;largem}} 
From Lemma \ref{lem:ionaumme} we obtain $c_2>0$ such that 
\[
 \io u^{m-1}(\cdot,t)\leq c_2 (1+t) \quad \text{and}\quad \intnt \io |\na u^{m-1}|^2 \leq c_2 (1+t) \fatntmax.
\]
 The fact that $p\in[m-1,\frac{2N}{N-2}(m-1)]$ entails both $\frac p{m-1}\geq 1$ and $\frac{p}{m-1}\leq \frac{2N}{N-2}$. Therefore, with 
\[
 a:=\frac{1-\frac{m-1}p}{1+\frac1N-\frac12}, 
\]
Lemma \ref{lem:GNI} \ref{GNI:firstderivative} produces $c_3>0$ such that for all $t\in(0,\Tmax)$ 
\begin{align*}
 \intnt \norm[\Lom p]{u}^r &= \intnt \norm[\Lom {\frac p{m-1}}]{u^{m-1}}^{\frac r{m-1}} \\
 &\leq c_3\intnt \norm[\Lom 2]{\na u^{m-1}}^{\frac{r}{m-1}a}\norm[\Lom 1]{u^{m-1}}^{\frac{r(1-a)}{m-1}} + c_3\intnt \norm[\Lom {\frac1{m-1}}]{u^{m-1}}^{\frac{r}{m-1}}\\
&\leq c_2\kl{c_3(1+t)}^{\frac{r(1-a)}{m-1}}\intnt\kl{\norm[\Lom2]{\na u^{m-1}}^2+1} + c_3\norm[\Lom 1]{u_0}^r t \leq c_4+c_5t^{r+1},
\end{align*}
where we have used that $\frac{ra}{m-1}\leq 2$ and, aided by Lemma \ref{lem:massconservation:u} and the trivial inequality $r\frac{1-a}{m-1}\leq r$, chosen suitable positive constants $c_4$ and $c_5$. 
\end{proof}

As preparation for exploiting \eqref{eq:ulpr} in the second equation of \eqref{sys}, we recall 
\begin{lemma}\label{lem:maxSobolev}
 Let $p,q\in(1,∞)$. Then for every $T>0$ there exists $C>0$ such that 
for every $z\in \LT q {\Lom p}$ the unique solution of 
\[
 v_t=Δv-z \;\;\text{ in } \Om\times(0,T),\quad \delny v\amrand=0, \quad v(\cdot ,0)=0
\]
satisfies 
\[
 \intnT \norm[\Lom p]{Δv}^q\leq C \intnT  \norm[\Lom p]{z}^q.
\]
\end{lemma}

\begin{proof}
We obtain this lemma as straightforward consequence of well-known maximal regularity assertions, cf. \cite{giga_sohr},\cite{hieber_pruess}.
\end{proof}
Lemma \ref{lem:maxSobolev} empowers us to develop \eqref{eq:ulpr} into useful knowledge about the gradient of $v$: 

\begin{lemma}\label{lem:upr.gives.navq}
 Let $p\geq \frac N2$, $r\geq p$, $(2-\frac Np)r>N$, and 
\begin{align*}
 \begin{cases} q\in(1,N+(2-\frac Np)r], &\text{if } p\geq N\\
  q\in(1,N+(2-\frac Np)r]\cap(1,\frac{Np}{N-p})&,\text{if } \frac{N}2<p<N.
 \end{cases}
\end{align*}
Then for every $K>0$ and $T>0$ there is $C>0$ such that for every $v_0\in W^{1,\infty}(\Om)$ with $\norm[W^{1,∞}(\Om)]{v_0}\leq K$, and every $u\in L^r((0,T);\Lom p)$ for which
\begin{equation}\label{eq:condintntlpu}
 \intnT \norm[\Lom p]{u}^r < K 
\end{equation}
is satisfied, the solution $v$ of \eqref{sys:v} fulfils
\begin{equation}\label{eq:ionavq}
 \intnT\io |\na v|^q<C.
\end{equation}
\end{lemma}
\begin{proof}
In order to prepare the application of Lemma \ref{lem:maxSobolev}, we decompose $v(\cdot,t)=\vtilde(\cdot,t)+e^{t\Delta}v_0$ in $\Om\times(0,T)$, where $\vtilde$ solves 
\[
 \vtilde_t=\Delta \vtilde - uv, \quad \vtilde(\cdot,0)=0, \quad \delny \vtilde\amrand=0.
\]
By nonnegativity of $v_0$ and $uv$, we clearly have $0\leq \vtilde\leq v\leq K$ in $\Om\times(0,T)$.\\
We let $q\leq N+(2-\frac Np)r$ and without loss of generality assume $q\geq 2p$ (which is possible since $2p=N+2p-N\leq N+(2p- N)\frac{r}{p}=N+(2-\frac Np)r$ and also $2p< \frac{Np}{N-p}$ if $p\in(\frac N2,N)$). We note that $q\leq N+(2-\frac{N}p)r$ implies that $r\geq \frac{q-N}{2-\frac Np}=\frac{(q-N)p}{2p-N}$ and hence with 
\[
 a:= \frac{p(q-N)}{q(2p-N)} 
\]
we have $aq\leq r$. 
Moreover, $q\geq2p$ ensures that $pq-Np\geq pq-\frac{Nq}2=\frac12 q(2p-N)$ and thus $a\geq \frac12$, and, furthermore, $(p-N)q>-Np$, which is obvious for $p>N$ and holds by assumption on $q$ if $p<N$, entails $2pq-Nq> pq-Np$ and hence $a< 1$. Accordingly, from \cite[Lemma 1.3 iii)]{win_aggregationvs} and the \GNI\ (Lemma \ref{lem:GNI} \ref{GNI:secondderivative}) we obtain $c_1>0$, $c_2>0$, respectively, such that we have 
\begin{align*}
 \intnT\io |\na v|^q &\leq 2^q\intnT\io |\na e^{t\Delta}v_0|^q + 2^q\intnT\io |\na \vtilde|^q\\
 &\leq c_2T\norm[\Lom q]{\na v_0}^q +2^q\intnT \norm[\Lom q]{\na \vtilde}^q\\
 &\leq c_2T |\Om|^{\frac1q}\norm[\Lom \infty]{\na v_0} + c_2 \intnT\norm[\Lom p]{Δ\vtilde}^{aq}\norm[\Liom]{\vtilde}^{(1-a)q}+ c\intnT\norm[\Liom]{\vtilde}^q.
\end{align*}
Since $aq<r$, due to \YI\ and boundedness of $\vtilde$ this estimate can be turned into 
\[
 \intnT\io |\na v|^q\leq c_3+c_4\intnT\norm[\Lom p]{Δ\vtilde}^r,
\]
for some $c_3>0$, $c_4>0$, where we may invoke the maximal Sobolev result of Lemma \ref{lem:maxSobolev} for $z=uv$ and hence 
\(\intnT  \norm[\Lom p]{z}^r\leq K^r \intnT\norm[\Lom p]{u}^r\) to conclude \eqref{eq:ionavq} from \eqref{eq:condintntlpu}.
\end{proof}


Another consequence of \eqref{eq:ulpr} is (local-in-time) boundedness of $w$:

\begin{lemma}\label{lem:wbd.ifupr} Assume that $r\in(1,∞)$, $p\in[1,∞)$ are such that $\frac{Nr}{2p(r-1)}<1$. Then for every $K>0$ there is $C>0$ such that whenever, for some $T>0$, $w\in\usualspace[T]$ solves \eqref{uw:w}, \eqref{uw:rand}, \eqref{uw:init} for some $w_0$ as in \eqref{uw-init} and some $u\in \usualspace[T]$ such that 
$\norm[\Liom]{w_0}\leq K$, $\frac{1}{|\Om|}\io u(\cdot,t)\leq K$ on $(0,T)$ and moreover 
\[
 \intnT \norm[\Lom p]{u}^r < K,
\]
then 
\[
 w(x,t) \leq C(1+t) \qquad \text{for all } (x,t)\in \Om\times (0,T). 
\]
\end{lemma}
\begin{proof}
 By nonpositivity of $-|∇w|^2$, we have that $0\leq w\leq \wtilde$, where $\wtilde$ solves 
\[
 \wtilde_t=Δ\wtilde + u,\quad \delny \wtilde\amrand=0,\quad  \wtilde(\cdot ,0)=w_0.
\]
For this function we can estimate 
\begin{equation}\label{eq:wtilde:duhamel}
 \norm[\Liom]{\wtilde(\cdot ,t)} \leq \norm[\Liom]{w_0} + \intnt \norm[\Liom]{e^{(t-s)Δ}\kl{u(\cdot ,s)-\ubar}} ds + \ubar \cdot t \quad \text{for all } t\in(0,T), 
\end{equation}
where $\ubar = \frac1{|\Omega|} \io u(\cdot ,t)\leq K$ 
For assessing the integral in \eqref{eq:wtilde:duhamel} we invoke \cite[Lemma 1.3 i)]{win_aggregationvs} to obtain $c_1>0$ such that 
\begin{align}\label{eq:normwtilde}
 &\intnt \norm[\Liom]{e^{(t-s)Δ} \kl{u(\cdot ,s)-\ubar}} ds\nn\\
 &\leq  c_1 \intnt \kl{1+(t-s)^{-\frac N{2p}}}e^{-λ_1(t-s)}\norm[\Lom p]{u(\cdot ,s)-\ubar}\\\
 &\leq  c_1 \intnt \kl{1+(t-s)^{-\frac N{2p}}} e^{-λ_1(t-s)}\norm[\Lom p]{u(\cdot ,s)} ds + c_1 K |Ω|^{\frac1p}\intninf \kl{1+σ^{-\frac N{2p}}}e^{-λ_1σ} dσ\nn\\
 & \leq c_1\intninf \kl{1+σ^{-\frac N{2p}}}^{\frac{r}{r-1}} e^{-λ_1σ\frac{r}{r-1}} dσ + c_1\intnt \norm[\Lom p]{u(\cdot ,s)}^r + c_1 K |Ω|^{\frac1p}\intninf \kl{1+σ^{-\frac N{2p}}}e^{-λ_1σ} dσ\nn
\end{align}
for all $t\in(0,T)$. 
Collecting the constants in \eqref{eq:wtilde:duhamel} and \eqref{eq:normwtilde}, we see that for all $t\in(0,T)$
\[
 w(x,t)\leq \norm[\Liom]{\wtilde(\cdot ,t)}\leq C(1+t), 
\]
where 
\begin{align*}
 &C:= K+c_1K+c_1K|\Om|^{\frac1p} \intninf \kl{1+σ^{-\frac N{2p}}}e^{-λ_1σ} dσ+k_1 \intninf \kl{1+σ^{-\frac N{2p}}}^{\frac{r}{r-1}} e^{-λ_1σ\frac{r}{r-1}} dσ,
\end{align*}
which is finite due to $\frac{Nr}{2p(r-1)}<1$ (and its consequence $\frac{N}{2p}<1$).
\end{proof}

If we can find parameters that allow for an application of Lemma \ref{lem:ulpr} and Lemma \ref{lem:wbd.ifupr} at the same time, we can conclude boundedness of $w$. This is the goal we pursue in the following lemma: 
\begin{lemma}\label{lem:wbd}
 Let 
\begin{equation}\label{eq:lem:mgeqn4}
 m>1+\frac N4 
\end{equation}
and $δ>0$. Then for all $T\in(0,\infty)$ there is $C>0$ such that for 
every $D\in \Cdm$ and every $(u_0,w_0)$ as in \eqref{uw-init} with $\norm[\Lmme]{u_0}\leq K$, $\norm[\Lom\infty]{w_0}\leq K$, any solution $(u,w)\in(\usualspace[T])^2$ of \eqref{eq:uw-system} satisfies 
\[
 w(x,t)\leq C \qquad \text{for all } x\in\Om \text{ and all } t\in(0,T). 
\]
\end{lemma}
\begin{proof}
Let us first consider the case $m\in(2-\frac1N,2]$ (that is of interest only if $N<4$, because $m$ is supposed to satisfy $m>1+\frac N4$) and observe that by \eqref{eq:lem:mgeqn4}, we have 
\[
 m> \left.\begin{cases}\frac32,&\text{if } N=2\\\frac74,&\text{if } N=3
\end{cases}\right\}= \frac54 - \frac1{2N} + \frac N8 +\sqrt{\kl{\frac54 -\frac1{2N}+\frac{N}8}^2-\frac54+\frac1N-\frac38N}.
\]
Therefore, we see that  
\[
 4m^2-10m+\frac4N m -Nm+5-\frac4N+\frac32N>0
\]
and hence 
\[
 (N-2)\kl{m-\frac32}=Nm-\frac32N-2m+3 < 4m^2-8m+\frac4N m-4m+8-\frac4N=2(m-1)\kl{2m-4+\frac2N}\!, 
\]
so that 
\[
 \frac{N(m-\frac32)}{2m-4+\frac2N} < \frac{2N}{N-2}(m-1). 
\]
Since moreover $\frac{2N}{N-2}(m-1)>1$, it is possible to choose $p\geq 1$ such that $p\in\kl{\frac{N(m-\frac32)}{2m-4+\frac2N},\frac{2N}{N-2}(m-1)}$. With this choice of $p$ we let 
\[
 r:=\frac{2m-3+\frac2N}{1-\frac1p}
\]
and note that $2m-3+\frac2N>4-\frac2N-3+\frac2N=1>1-\frac1p$ entails $r>1$. Hence Lemma \ref{lem:ulpr} \ref{ulpr;smallm} is applicable. Moreover, 
\[
 \frac{2p}N\kl{1-\frac1r}=\frac{2p}N\kl{1-\frac{1-\frac1p}{2m-3+\frac2N}}=\frac2N\cdot\frac{p(2m-3+\frac2N)-p+1}{2m-3+\frac2N}>\frac2N\cdot\frac{N(m-\frac32)+1}{2(m-\frac32+\frac1N)}=1
\]
and we can additionally invoke Lemma \ref{lem:wbd.ifupr} so as to obtain the desired boundedness of $w$ on $\Om\times(0,T)$.\\
If $m\geq 2$ (and $m> 1+\frac N4$), we note that 
\[
 \frac{N^2(m-1)}{2N(m-1)+4(m-1)-2N}< \frac{N^2(m-1)}{2N\frac N4+4\frac{N}4-2N}=\frac{N^2(m-1)}{\frac{N^2}2-N}=\frac{2N}{N-2}(m-1).
\]
Since $m\geq 2$, 
\[
 \frac{1+\frac2N}{m-1}\leq 1+\frac2N < 1+\frac4N+\frac4{N^2}, 
\]
and hence 
\[
 \frac1{m-1}-1-\frac2N < \frac2N+\frac4{N^2}-\frac2{N(m-1)}=\frac{2N(m-1)+4(m-1)-2N}{N^2(m-1)}.
\]
Therefore we can pick $p\in\kl{\frac{N^2(m-1)}{2N(m-1)+4(m-1)-2N},\frac{2N}{N-2}(m-1)}$ such that $\frac1p>\frac1{m-1}-1-\frac2N$ and $p> m-1$, and we let $r:=\frac{1+\frac2N}{\frac1{m-1}-\frac1p}$. 
Then $r>1$ and, apparently, $(\frac1{m-1}-\frac1p)r\leq 1+\frac2N$, warranting applicability of Lemma \ref{lem:ulpr}.
Moreover, $p>\frac{N^2(m-1)}{2N(m-1)+4(m-1)-2N}$ entails $\frac1p<\frac2N+\kl{\frac2N}^2-\frac2{N(m-1)}$ and thus $\frac{N}{2p}(1+\frac2N)=\frac{N}{2p}+\frac1p<1+\frac2N-\frac1{m-1}+\frac1p$ and hence, finally, 
\[
 \frac{N}{2p}<1-\frac{\frac1{m-1}-\frac1p}{1+\frac2N}=1-\frac1r,
\]
which permits us to employ Lemma \ref{lem:wbd.ifupr} and conclude.
\end{proof}

\begin{lemma}\label{lem:nav.gives.naw}
 For every $K>0$ and every $q\in(0,\infty]$ there is $C>0$ such that for all $T>0$ and all $v\in \usualspace[T]$ 
\[
 \norm[\Lom\infty]{v_0}\geq \frac1K,\quad w\leq K \quad \text{ in } \Om\times(0,T), \quad \text {and}\quad \norm[L^q(\Om\times(0,T))]{\na v}\leq K
\]
implies
\[
 \norm[L^q(\Om\times(0,T))]{\na w}\leq C 
\]
\end{lemma}
\begin{proof}
Since $w\leq K$, we have $v=\norm[\Liom]{v_0} e^{-w}\geq \norm[\Liom]{v_0}e^{-K}$, and immediately obtain $\frac1v\leq \norm[\Liom]{v_0}^{-1}e^K\leq Ke^K$ in $\Om\times(0,T)$. Thus  
\begin{equation*}
 \norm[L^q(\Om\times(0,t))]{\na w}\leq \norm[L^q(\Om\times(0,t))]{\frac1v \na v} \leq Ke^K \norm[L^q(\Om\times(0,t))]{\na v}\leq K^2e^K=:C. \qedhere
\end{equation*}
\end{proof}

\begin{lemma}\label{lem:naw.q.gives.u.p}
 Let $δ>0$, $m\geq1$, $q>2$ and $p>1$. Then for every $K>0$ and $T>0$ there is $C>0$ such that the following holds:
If $q\geq N$ 
and
\begin{align}
m&\leq 2,& p&\geq m-\frac2q,& p&\leq (q-1)(m-1)+\frac{q-2}N, \label{eq:intupcondmsmall}\\
\text{or}\qquad  m&\geq 2,& p&\geq 2\kl{1-\frac1q}(m-1),& p&\leq (m-1)\kl{\frac q2+\frac{(q-2)(N+2)}{2N}},\label{eq:intupcondmlarge}
\end{align}
then for every function $w\in\usualspace[T]$ with
\[
\intnt\io |\na w|^q \leq K \text{ for all  } t\in (0,T),
\]
any solution $u\in\usualspace[T]$ of \eqref{uw:u}, \eqref{uw:rand}, \eqref{uw:init} with $\norm[\Lmme]{u_0}\leq K$ and some $D\in \Cdm$
fulfils
\[
\io u^p(\cdot,t)\leq C \qquad \text{for all } t\in(0,T).
\]
\end{lemma}
\begin{proof}
 Either of \eqref{eq:intupcondmsmall} and \eqref{eq:intupcondmlarge} implies $p\geq m-1$. Moreover, 
\begin{equation}\label{eineGNIvoraussetzung}
 \frac{N-2}{2N}\leq \frac{q-2}{2q} \cdot \frac{p+m-1}{p-m+1}.
\end{equation}
Let us first consider the case $m\leq 2$. Then \(p\geq m-\frac2q\) implies \(p-m+1\geq\frac{q-2}q\) and hence
\begin{equation}\label{dieandereGNIvoraussetzung:a}
 \frac{2}{m+p-1}\leq \frac{2q}{q-2}\cdot\frac{p-m+1}{p+m-1}.
\end{equation}
We now let
\[
a:=\frac{\frac{m+p-1}{2}-\frac{q-2}{2q}\cdot\frac{p+m-1}{p-m+1}}{\frac{m+p-1}{2}+\frac1N-\frac12}
\]
and observe that
\begin{equation}\label{atimesleq2}
 a\cdot \frac{2q}{q-2}\cdot\frac{p-m+1}{m+p-1}\leq 2, 
\end{equation}
because
\(p\leq (q-1)(m-1)+\frac{q-2}N\) implies that \(\frac{q}{q-2}p-p=(\frac{q}{q-2}-1)p=\frac{2}{q-2} p \leq \frac{2(q-1)}{q-2} (m-1)+\frac{2}{N}=(1+\frac q{q-2})(m-1)+\frac{2}{N}=m-1+\frac{q}{q-2}(m-1)+\frac2N\), that is, $\frac{q}{q-2}(p-m+1)\leq m+p-1+\frac2N\) and hence \(\frac{q}{q-2}(p-m+1)-1\leq \kl{m+p-1}+\frac2N-1\), which leads to 

\[ a\cdot \frac{2q}{q-2}\cdot \frac{p-m+1}{m+p-1} = \frac{\frac{m+p-1}{2}-\frac{q-2}{2q}\frac{p+m-1}{p-m+1}}{\frac{m+p-1}{2}+\frac1N-\frac12} \cdot \frac{2q}{q-2}\cdot \frac{p-m+1}{m+p-1}=\frac{\frac{p-m+1}2\cdot\frac{2q}{q-2}-1}{\frac12((m+p-1)+\frac2N-1)} \leq 2.
\]

From Lemma \ref{lem:massconservation:u} we obtain $c_1>0$ such that
\[
\norm[\Lom {\frac2{m+p-1}}]{u^{\frac{m+p-1}2}(\cdot,t)} = c_1 \qquad \text{for all } t\in(0,T).
\]
Due to \eqref{eineGNIvoraussetzung} and \eqref{dieandereGNIvoraussetzung:a} we can apply the \GNI\ in the form of Lemma \ref{lem:GNI} \ref{GNI:firstderivative} to obtain $c_2>0$ such that
\begin{align}\label{intu.fuer.a}
 \io u^{(p+1-m)\frac{q}{q-2}} &= \io u^{\frac{m+p-1}2 ( \frac{2q}{q-2}\cdot \frac{p-m+1}{m+p-1})} \nn\\
 &= \norm[\Lom { \frac{2q}{q-2}\cdot \frac{p-m+1}{m+p-1}}]{u^{\frac{m+p-1}2}}^{ \frac{2q}{q-2}\cdot \frac{p-m+1}{m+p-1}}\nn\\
 &\leq c_2 \norm[\Lom 2]{∇u^{\frac{m+p-1}2}}^{a\cdot \frac{2q}{q-2}\cdot \frac{p-m+1}{m+p-1}} \norm[\Lom {\frac{2}{m+p-1}}]{u^{\frac{m+p-1}2}}^{(1-a)\cdot \frac{2q}{q-2}\cdot \frac{p-m+1}{m+p-1}}+c_2\norm[\Lom{\frac{2}{m+p-1}}]{u^{\frac{m+p-1}2}}^{ \frac{2q}{q-2}\cdot \frac{p-m+1}{m+p-1}}\nn\\
 &= c_2 c_1^{(1-a)\cdot \frac{2q}{q-2}\cdot \frac{p-m+1}{m+p-1}}\norm[\Lom 2]{∇u^{\frac{m+p-1}2}}^{a\cdot \frac{2q}{q-2}\cdot \frac{p-m+1}{m+p-1}} + c_2c_1^{ \frac{2q}{q-2}\cdot \frac{p-m+1}{m+p-1}}
\end{align}
on $(0,T)$.\\
In obtaining such an estimate for $m\geq 2$ we could use the same argument. It is, however, possible to obtain better conditions by relying on Lemma \ref{lem:intumme} instead of Lemma \ref{lem:massconservation:u}. Apart from that, the reasoning is analogous:
We have \(p\geq 2\big(1-\frac1q\big)(m-1)\), which implies \(qp\geq (q-2+q)(m-1)$, thus $q(p-m+1)\geq(m-1)(q-2)\) and hence 
\begin{align}\label{dieandereGNIvoraussetzung:b}
 \frac{2(m-1)}{m+p-1}\leq \frac{2q}{q-2}\cdot\frac{p-m+1}{p+m-1}
\end{align}
and let
\[
b:=\frac{\frac{m+p-1}{2(m-1)}-\frac{q-2}{2q}\cdot\frac{p+m-1}{p-m+1}}{\frac{m+p-1}{2(m-1)}+\frac1N-\frac12},
\]
noting that
\begin{equation}\label{bleq}
 b\cdot \frac{2q}{q-2}\cdot \frac{p-m+1}{m+p-1} \leq 2,
\end{equation}
because \(p\leq (m-1)(\frac q2+\frac{(q-2)(N+2)}{2N})\) implies that \((\frac{q}{q-2}-1)(\frac{p}{m-1})=\frac{2}{q-2} \frac{p}{m-1}\leq \frac{N+2}N+\frac q{q-2}\) and hence \(\frac{q(p-m+1)}{(m-1)(q-2)}=\frac{q}{q-2}(\frac{p}{m-1}-1)\leq \frac {p}{m-1} + \frac{N+2}N=\frac{m+p-1}{m-1}+\frac2N\), which shows that \(\frac{p-m+1}{2(m-1)}\cdot \frac{2q}{q-2}-1\leq \frac{m+p-1}{m-1}+\frac2N-1\) and therefore also 
\begin{align*}
b\cdot \frac{2q}{q-2}\cdot \frac{p-m+1}{m+p-1} =\frac{\frac{m+p-1}{2(m-1)}-\frac{q-2}{2q}\cdot\frac{p+m-1}{p-m+1}}{\frac{m+p-1}{2(m-1)}+\frac1N-\frac12}\cdot \frac{2q}{q-2}\cdot \frac{p-m+1}{m+p-1}\\=\frac{\frac{q}{q-2}\cdot\frac{p-m+1}{2(m-1)}-1}{\frac{m+p-1}{2(m-1)}+\frac1N-\frac12}\leq 2.
\end{align*}
Lemma \ref{lem:intumme} yields $c_3>0$ such that
\[
\norm[\Lom {\frac{2(m-1)}{m+p-1}}]{u^{\frac{m+p-1}2}(\cdot,t)} \leq c_3\qquad \text{for all } t\in(0,T)
\]
and hence \eqref{eineGNIvoraussetzung} and \eqref{dieandereGNIvoraussetzung:b} enable us to invoke the \GNI\ and obtain $c_4>0$ such that on $(0,T)$ 
\begin{align}\label{intu.fuer.b}
 \io u^{(p+1-m)\frac{q}{q-2}} &= \io u^{\frac{m+p-1}2 ( \frac{2q}{q-2}\cdot \frac{p-m+1}{m+p-1})} \nn\\
 &= \norm[\Lom { \frac{2q}{q-2}\cdot \frac{p-m+1}{m+p-1}}]{u^{\frac{m+p-1}2}}^{ \frac{2q}{q-2}\cdot \frac{p-m+1}{m+p-1}}\nn\\
 &\leq c_4 \norm[\Lom 2]{∇u^{\frac{m+p-1}2}}^{a\cdot \frac{2q}{q-2}\cdot \frac{p-m+1}{m+p-1}} \norm[\Lom {\frac{2(m-1)}{m+p-1}}]{u^{\frac{m+p-1}2}}^{(1-a)\cdot \frac{2q}{q-2}\cdot \frac{p-m+1}{m+p-1} }+c_4\norm[\Lom{\frac{2}{m+p-1}}]{u^{\frac{m+p-1}2}}^{ \frac{2q}{q-2}\cdot \frac{p-m+1}{m+p-1}}\nn\\
 &\leq c_4 c_3^{(1-a)\cdot \frac{2q}{q-2}\cdot \frac{p-m+1}{m+p-1}}\norm[\Lom 2]{\nabla u^{\frac{m+p-1}2}}^{a\cdot \frac{2q}{q-2}\cdot \frac{p-m+1}{m+p-1}}+c_4c_3^{\frac{2q}{q-2}\cdot \frac{p-m+1}{m+p-1}}.
\end{align}

From either \eqref{intu.fuer.a} and \eqref{atimesleq2} or \eqref{intu.fuer.b} and \eqref{bleq} (and possibly \YI) we hence find that with some $c_5>0$ we have
\begin{equation}\label{eq:resultofmanyGNIs}
\io u^{(p+1-m)\frac{q}{q-2}} \leq c_5\norm[\Lom2]{u^{\frac{m+p-3}2}\na u}^2 + c_5 \qquad \text{on } (0,T).
\end{equation}
In 
\[
 \frac1p \ddt \io u^p + (p-1)δ\io u^{p+m-3}|∇u|^2 \leq (p-1)\left\lvert \io u^{p-1}∇u\cdot ∇w\right\rvert\qquad \text{on } (0,T)
\]
we can apply Young's inequality to see that on $(0,T)$ 
\[
 (p-1)\left\lvert \io u^{p-1}∇u\cdot∇w\right\rvert \leq \frac{(p-1)δ}4 \io u^{p+m-3}|∇u|^2 + \frac{p-1}{δ} \io u^{p-m+1}|∇w|^2.
\]
A further application of Young's inequality allows us to separate $u$ and $|\na w|$ in the last integral according to 
\[
 \frac{p-1}{δ}\io u^{p-m+1} |∇w|^2 \leq \frac{c_5(p-1)}{δ^3} \io |∇w|^q + \frac{(p-1)δ}{4c_5} \io u^{(p+1-m)\frac{q}{q-2}}, \quad \text{on } (0,T).
\]
Therefore, due to \eqref{eq:resultofmanyGNIs}, in total, 
\[
 \frac1p \ddt \io u^p + \frac {(p-1)δ}2\io u^{p+m-3}|∇u|^2 \leq \frac{(p-1)δ}4 + \frac{c_5(p-1)}{δ^3} \io |∇w|^q\qquad \text{on } (0,T). 
\]
Integration with respect to time produces the lemma. 
\end{proof} 

We are particularly interested in applying the previous lemma for some $p>N$, because for such $p$, a bound on $\io u^p$ on some interval $[0,T]$ already ensures uniform boundedness of $\na v$ (and hence $\na w$) on $\Ombar\times[0,T]$. 

\begin{lemma}\label{lem:nawq.gives.up.for.pgeqn}
Let $δ>0$. Assume that either 
\beenum
\item \label{nawq.gives.up.for.pgeqn;smallm}
$2-\frac1N< m\leq 2$, $N\geq 2$, $q>N$ and $q>1+\frac{N^2+1}{Nm-N+1}$, or
\item \label{nawq.gives.up.for.pgeqn;largem}
$m\geq 2$, $N\geq 2$, $q>N$ and $q>\frac{2n^2+2m^2+2m-4}{(m-1)(N+m+2)}$.
\end{enumerate}
Then there is $p>N$ and for every $K>0$ and $T\in(0,\infty)$ there is $C>0$ such that whenever $u\in\usualspace[T]$ solves \eqref{uw:u}, \eqref{uw:rand}, \eqref{uw:init}, with some $D\in\Cdm$, some $u_0$ as in \eqref{uw-init} and such that $\norm[\Lmme]{u_0}\leq K$, and some $w\in \usualspace[T]$ satisfying 
\[
 \intnT \io |\na w|^q \leq K, 
\]
then 
\[
 \io u^p(\cdot,t) \leq C \qquad \text{for all } t\in(0,T).
\]
\end{lemma}
\begin{proof}
\textbf{\ref{nawq.gives.up.for.pgeqn;smallm}} For $\qtilde=2$ we have $m-\frac2\qtilde=m-1=(\qtilde-1)(m-1)+\frac{\qtilde-2}N$ and because $m> 2-\frac1N$, for every $\qtilde\geq2$ we have 
\[
 \frac{d}{d\qtilde} \kl{m-\frac2\qtilde}=\frac2{\qtilde^2}\leq 1= 2-\frac1N -m+m-1+\frac1N< m-1+\frac1N=\frac{d}{d\qtilde}\kl{(\qtilde-1)(m-1)+\frac{\qtilde-2}m}.
\]
Therefore $m-\frac2q< (q-1)(m-1)+\frac{q-2}N$.
Furthermore \(
 q>1+\frac{N^2+1}{mN-N+1}=\frac{m-1+\frac2N+N}{m-1+\frac1N}\) implies that \[(q-1)(m-1)+\frac{q-2}N=q\kl{m-1+\frac1N}+1-m-\frac2N>N.\] 
Hence it is possible to find $p>N$ such that $p>m-\frac2q$ and $p<(q-1)(m-1)+\frac{q-2}N$ and an application of Lemma \ref{lem:naw.q.gives.u.p} proves the statement.\\
\textbf{\ref{nawq.gives.up.for.pgeqn;largem}} Since $x+\frac1x\geq 2$ for all $x>0$, and since $q\geq 2$, we have 
\[
  2-\frac2q\leq \frac q2+\frac{(q-2)(m+2)}{2N}
\]
and hence $2(1-\frac1q)(m-1)\leq (m-1)(\frac q2+\frac{(q-2)(m+2)}{2N})$. The fact that $q>\frac{2N^2+2m^2+2m-4}{(m-1)(N+m+2)}=\frac{1}{(m-1)(N+m+2)}(2N^2+(2m+4)(m-1))=(\frac{2N^2}{m-1}+2m+4)\frac1{N+m+2}$ shows that $q(N+m+2)>2m+4+\frac{2N^2}{m-1}$ and hence $N<\frac{m-1}{2N}(q(N+m+2)-2m-4)=\frac{m-1}{2N}(Nq+(q-2)(m+2))=(m-1)(\frac q2+\frac{(q-2)(m+2)}{2N})$.
Therefore we can choose $p>N$ such that 
\[
 p<(m-1)\kl{\frac q2 + \frac{(q-2)(m+2)}{2N}} \quad \text{ and } \quad p>2\kl{1-\frac1q}(m-1)
\]
and apply Lemma \ref{lem:naw.q.gives.u.p} for this choice of $p$ to obtain the assertion. 
\end{proof}

The previous lemma requires a bound on some $\intnT\io |\na w|^q$. Fortunately, this is exactly what we have prepared in Lemma \ref{lem:ulpr}, Lemma \ref{lem:upr.gives.navq}, Lemma \ref{lem:wbd}, and Lemma \ref{lem:nav.gives.naw}.

\begin{lemma}\label{lem:intupbounded}
 Let $m>1+\frac N4$ and $δ>0$. Then there is $p>N$ and for every $K>0$ and $T>0$ there is $C>0$ such that every solution $(u,w)\in(\usualspace[T])^2$ of \eqref{eq:uw-system} 
 with initial data $(u_0,w_0)$ as in \eqref{uw-init} and with $\norm[\Lom 1]{u_0}\leq K$, $\norm[W^{1,\infty}(\Om)]{w_0}\leq K$ and any $D\in \Cdm^+$ satisfies 
\[
 \int u^p(\cdot,t)\leq C \qquad \text{for every } t\in(0,T).
\]
\end{lemma}
\begin{proof}
By the choice of $m$, from Lemma \ref{lem:wbd} we know that we can find $C>0$ such that for any $u_0$, $w_0$ and $D$ as above, any solution $(u,w)\in(\usualspace[T])^2$ of \eqref{eq:uw-system} satisfies $0\leq w\leq C$ in $\Om\times(0,T)$. Lemma \ref{lem:nav.gives.naw} therefore warrants that the desired conclusion results from a combination of Lemma \ref{lem:ulpr} and Lemma \ref{lem:upr.gives.navq} with Lemma \ref{lem:nawq.gives.up.for.pgeqn} -- provided that there are parameters $p,q,r$ that simultaneously satisfy all conditions posed by these lemmata. This is what we ensure in the remainder of the proof: \\
\textbf{Case $N=2$, $m\in(\frac32,2]$:} We let $r=4(m-1)$, $p=2$, and $q=4m-2$. 
 Then $r(1-\frac1p)=4(m-1)(1-\frac12)=2(m-1)\leq 2m-2=2m-3+\frac22$, which enables us to invoke Lemma \ref{lem:ulpr} \ref{ulpr;smallm}. Moreover, $m>\frac32$ implies $4m-4> 2$ and thus $r> p$, and we have $q=4m-2=2+4m-4\leq 2+(2-\frac22)4(m-1)=N+(2-\frac Np)r$ as well as $(2-\frac Np)r=r>N$. Therefore, Lemma \ref{lem:upr.gives.navq} becomes applicable.
Thanks to $q=4m-2\geq 4\cdot\frac32-2=4>2=N$ and thanks to $m\geq \frac32$, hence $q=4m-2>4\cdot32-2>\frac72=1+\frac{5}{2\cdot\frac32-1}>1+\frac5{2m-1}=1+\frac{N^2+1}{Nm-N+1}$ holds true, facilitating the use of Lemma \ref{lem:nawq.gives.up.for.pgeqn} \ref{nawq.gives.up.for.pgeqn;smallm}.\\
\textbf{Case $N=3$, $m\in(\frac 74,2]$:} 
 Here we let $r=p=2m-\frac43$. Then $p\geq1$, $r\geq1$, $p=2m-\frac43<6m-6=\frac{2N}{N-2}(m-1)$ and $r(1-\frac1p)=r-1=2m-\frac73=2m-3+\frac2N$, so that Lemma \ref{lem:ulpr} can be used. Since $m>\frac74=\frac{21}{12}>\frac{19}{12}$, we have that $12m^2-19m-\frac83=12(m-\frac{19}{12})m-\frac83\geq 12(\frac74-\frac{19}{12})\frac74-\frac83=\frac72-\frac83=\frac{21-16}6>0$ and thus $3m+8 < 12m^2-16m+\frac{16}3 = (4m-\frac83)(3m-2)$, i.e. $2p> \frac{3m+8}{3m-2}$. Furthermore,  $p=2m-\frac43\leq 4-\frac43=\frac83<3$ and $p=2m-\frac43\geq \frac72-\frac43=\frac{21-8}6=\frac{13}6>\frac32$, so that consequently, also $\frac{3p}{3-p}<2p$ holds. We choose $q\in (\frac{3m+8}{3m-2},2p)$, thereby ensuring the applicability of Lemma \ref{lem:upr.gives.navq}.\\ 
Since finally $q>\frac{3m+8}{3m-2}=\frac{m+\frac 83}{m-\frac23}=1+\frac{\frac{10}3}{m-\frac23}=1+\frac{3+\frac13}{m-1+\frac13}$ and $q>\frac{3m+8}{3m-2}
=1+\frac{10}{3m-2}\geq 1+\frac{10}{6-2}=\frac72>3\geq2$ we may also draw on Lemma \ref{lem:nawq.gives.up.for.pgeqn} \ref{nawq.gives.up.for.pgeqn;smallm}.



\textbf{Case $N\geq 2$, $m\geq 2$, $m\geq 1+\frac N4$:} Let $r:=p:=2\frac{N+1}{N}(m-1)$. Then obviously $p=r>1$. Moreover, $p\leq \frac{2N}{N-2}(m-1)$ (because $\frac{2N}{N-2}>\frac{2+2N}{N}$ is equivalent to $2N^2>2N^2+2N-4N-4$ and hence to $0>-2N-4$) and 
\[
 \kl{\frac1{m-1}-\frac1p}r=\frac{p}{m-1}-1=2\frac{N+1}N-1=\frac{N+2}N\leq 1+\frac2N, 
\]
so that the conditions of Lemma \ref{lem:ulpr} \ref{ulpr;largem} are satisfied. 
We furthermore let $q:=2p=\frac{4(N+1)}N(m-1)$ and note that $p>\frac N2$, since $2\frac{N+1}N(m-1)>2\cdot \frac{N+1}N \frac N4=\frac{N+1}2>\frac N2$, and that $q\leq 2p=2r+N-N\frac rp$, that moreover either $p\geq N$ or $p<N$ and $q=2p<\frac{Np}{N-p}$, because $p>\frac N2$, and therefore Lemma \ref{lem:upr.gives.navq} is applicable.
In order to see that these choices also make the use of Lemma \ref{lem:nawq.gives.up.for.pgeqn} \ref{nawq.gives.up.for.pgeqn;largem} viable, we first investigate the polynomial
\[
 P_N(m):=(2N+2)m^3+(2N^2+N)m^2+(-4N^2-11N-6)m-N^3+2N^2+8N+4.
\]
It is extremal whenever $P_N'(m)=(6N+6)m^2+(4N^2+2N)m+(-4N^2-11N-6)=0$, which is the case for exactly two real numbers that lie in $(-∞,2)$, because for $m\geq 2$ we have $P_N'(m)\geq (24N+24)+(8N^2+4N)+(-4N^2-11N-6)>0$. 
We claim that $P_N(m)> 0$ for any $m> \max\set{2,1+\frac N4}$ and for this compute $P_N(\max\set{2,1+\frac N4})$: 
\begin{align*}
P_N(2)&=16N+16+8N^2+4N-8N^2-22N-12-N^3+2N^2+8N+4\\
&=-N^3+2N^2+6N+8\\
&=\begin{cases}-8+8+12+8>0,&N=2,\\-27+18+18+8>0,&N=3,\\-64+32+24+8=0,&N=4,\end{cases} 
\end{align*}
and 
\begin{align*}
 &P_N\kl{1+\frac N4}=\\
&\frac1{4^3}\bigg((2N\!+\!2)(N\!+\!4)^3\!+\!4(2N^2\!+\!N)(N\!+\!4)^2\!+\!16(\!-\!4N^2\!-\!11N\!-\!6)(N\!+\!4)\\&\qquad\qquad-64N^3\!+\!128N^2\!+\!512N\!+\!256\bigg)\\
&=\frac2{4^3}N^2(5N+3)(N-4), 
\end{align*}
which is nonnegative for $N\geq 4$. Since $P_N$ is nonnegative in $\max\set{2,1+\frac N4}$ and strictly increasing on $(2,∞)$, we conclude that $P_N(m)> 0$ for any $m>\max\set{2,1+\frac N4}$. Positivity of $P_N(m)$ is equivalent to 
\[
 2(N+1)(m-1)^2(N+m+2)>N^3+m^2N+mN-2N
\]
and hence 
\[
 q=\frac{4(N+1)}N(m-1)>\frac{2N^2+2m^2+2m-4}{(m-1)(N+m+2)}. 
\]
Furthermore by the fact that $p>\frac N2$, we also have $q>N$, and can invoke Lemma \ref{lem:nawq.gives.up.for.pgeqn} \ref{nawq.gives.up.for.pgeqn;largem}.
\end{proof}

Having completed the necessary preparations, we can now turn to the proof of existence of a global solution. In order to lay the groundwork for compactness arguments in Section \ref{sec:weaksol}, at the same time we derive a batch of estimates for the solutions.

\begin{lemma}\label{lem:bounds}
Let $δ>0$, $m>1+\frac{N}4$. 
\beenum
\item\label{bounds:exist}  For any $(u_0,v_0)$ as in \eqref{initcond} and any $D\in \Cdm^+$ there is a global classical solution $(u,v)\in(\usualspace[∞])^2$ to \eqref{sys}. 
\item \label{bounds:bounds} Moreover, for every $T>0$, $K>0$ there is $C_T>0$ such that for every $D\in\Cdm$ and $(u_0,v_0)$ as in \eqref{initcond} with $\norm[\Lmme]{u_0}\leq K$, $\frac1K\leq \norm[\Liom]{v_0}$, $\norm[W^{1,∞}(\Om)]{v_0}\leq K$, every solution $(u,v)\in(\usualspace[T])^2$ to \eqref{sys} satisfies 
\begin{align}
 \norm[L^{\infty}(\Om\times(0,T))]{u}&\leq C_T\label{bd:u}\\
 \norm[L^{\infty}((0,T);W^{1,\infty}(\Om))]{v}&\leq C_T\label{bd:v}\\
\norm[L^{\infty}(\Om\times(0,T))]{\frac1v\na v}&\leq C_T\label{bd:nav/v}\\
 \norm[L^2(\Om\times(0,T))]{D(u)\na u}&\leq C_T\label{bd:Dunau}\\
 \norm[L^2(\Om\times(0,T))]{\na u^{m-1}}&\leq C_T,\label{bd:naumme}\\
 \intnT\io D(u)u^{m-3}|\na u|^2&\leq C_T\label{bd:uDnau},\\
 \norm[L^{2}((0,T);(W_0^{1,1}(\Om))^\ast)]{v_t}&\leq C_T \label{bd:vt},\\
 \norm[L^1((0,T);(W_0^{1,N+1}(\Om))^\ast)]{u_t}&\leq C_T\bigg(1+\sup_{s\in[0,C_T]}D(s)\bigg)\label{bd:ut}.
\end{align}
\end{enumerate}
\end{lemma}
\begin{proof} According to Lemma \ref{lem:locex:and:continuation}, corresponding to $(u_0,v_0)$ and $D$ as in the hypothesis of the present lemma, there is a local solution $(u,v)\in(\usualspace)^2$. 
We now let $T\in(0,\Tmax]\cap(0,\infty)$ and $K>0$. By $\ID$ let us abbreviate the set of initial data 
\begin{align*}
\ID:=\bigg\{&(u_0,v_0)\in C^{α}(\Ombar)\times W^{1,∞}(\Om) \text{ for some } α\in(0,1); 
\norm[\Liom]{u_0}\leq K, \norm[W^{1,∞}(\Om)]{v_0}\leq K \bigg\}.
\end{align*}
Lemma \ref{lem:intupbounded} provides us with $p>N$ and $c_1>0$ such that for every $D\in\Cdm$ and every $(u_0,v_0)\in\ID$, every classical solution $(u,v)\in(\usualspace[T])^2$ of \eqref{sys} satisfies 
\[
 \norm[L^{\infty}((0,T);L^p(\Om))]{u}\leq c_1 
\]
and hence 
\[
 \norm[L^{\infty}(\Om\times(0,T))]{\na v}\leq c_2 \quad \text{and}\quad \norm[L^{\infty}(\Om\times(0,T))]{w}\leq c_3 
\]
as well as 
\[
 \norm[L^{\infty}(\Om\times(0,T))]{\na w}\leq c_4
\]
with some $c_2$, $c_3$ and $c_4$ obtained from Lemma \ref{lem:ex.estimates.parab} \ref{lem:exestimparab:nav}, Lemma \ref{lem:wbd} and Lemma \ref{lem:nav.gives.naw}, respectively, and with $w$ being defined as in \eqref{eq:def:w}. This asserts \eqref{bd:v} and \eqref{bd:nav/v}.  
An application of Lemma \ref{lem:naw.q.gives.u.p} for sufficiently large values of $q$ and $p$ then ascertains the existence of $c_5>0$ such that for all $D\in\Cdm$ and all $(u_0,v_0)\in\ID$ any classical solution $(u,v)$ of \eqref{sys} satisfies 
\[
 \norm[L^\infty((0,T);L^{N+3}(\Om))]{u\na w} \leq c_5, 
\]
again with $w$ as in \eqref{eq:def:w}. Additionally taking into account Lemma \ref{lem:massconservation:u}, we can apply Lemma \ref{lem:ex.estimates.parab} \ref{lem:exestimparab:taowin} with $f:=u\na w$ so as to obtain $c_6>0$ such that for all $D\in \Cdm$ and all $(u_0,v_0)\in\ID$ every classical solution $(u,v)$ of \eqref{sys} satisfies 
\[
 \norm[L^\infty(\Om\times(0,T))]{u}\leq c_6, 
\]
 which shows \eqref{bd:u} and -- in light of the extensibility criterion in \eqref{eq:extcrit} -- also proves \ref{bounds:exist}.
Given $D\in \Cdm$ we let $\Dbar(s):=\int_0^s D(σ) dσ$ and $\Dbarbar(s):=\int_0^s \Dbar(σ)dσ$ for $s>0$. Then for every $D\in \Cdm$ and $(u_0,v_0)\in\ID$, any classical solution $(u,v)$ of \eqref{sys} obeys $u_t=\Delta \Dbar(u)+\na\cdot(u\na w)$ with $w$ as in \eqref{eq:def:w}, and testing this equation by $\Dbar(u)$ 
we obtain 
\[
 \intnT\io (\Dbarbar(u))_t = -\intnT\io |\na \Dbar(u)|^2-\intnT\io u\na w\cdot \na\Dbar(u), 
\]
which, by Young's inequality, turns into 
\[
 \io \Dbarbar(u(\cdot,T)) + \frac12 \intnT\io |D(u)\na u|^2\leq \io \Dbarbar(u_0) +\frac12 \intnT\io u^2|\na w|^2\leq |\Om|\Dbarbar(K) + \frac12|\Om|Tc_6^2c_4^2,
\]
due to nonnegativity of $D$ proving \eqref{bd:Dunau}. 
The existence of $c_7>0$, $c_8>0$ such that for any $D\in \Cdm$ and any $(u_0,v_0)\in \ID$ any solution of \eqref{sys} satisfies 
\begin{align*}
 \norm[L^2(\Om\times(0,T))]{\na u^{m-1}}&\leq c_7,&\qquad   \intnT\io D(u)u^{m-3}|\na u|^2&\leq c_8
\end{align*}
immediately results from Lemma \ref{lem:ionaumme}, so that \eqref{bd:naumme} and \eqref{bd:uDnau} have been shown. For every $φ\in C_0^\infty(\Om)$ we have that any solution $(u,v)$ of \eqref{sys} for any $D\in \Cdm$, $(u_0,v_0)\in\ID$ satisfies 
\[
 \betr{\io v_tφ}=\betr{-\io \naφ\cdot \na v-\io uvφ}\leq c_2\norm[\Lom 1]{\na φ}+ Kc_6\norm[\Lom1]{φ}
\]
and we can conclude \eqref{bd:vt}.
We let $c_9>0$ be such that $\norm[\Liom]{ϕ}\leq c_9$ for every $ϕ\in W_0^{1,N+1}(\Om)$ with $\norm[W_0^{1,N+1}(\Om)]{ϕ}\leq 1$ and $c_{10}>0$, $c_{11}>0$ such that $\norm[L^2(\Om\times(0,T))]{ϕ}\leq c_{10}$, $\norm[L^1(\Om\times(0,T))]{ϕ}\leq c_{11}$ for every $ϕ\in L^\infty((0,T);L^{N+1}(\Om))$ with $\norm[L^\infty((0,T);W_0^{1,N+1}(\Om))]{ϕ}\leq 1$. We denote $X:=L^1((0,T);(W_0^{1,N+1}(\Om))^\ast)$ and thus have $X^\ast=L^\infty((0,T);W_0^{1,N+1}(\Om))$. Taking $φ\in X^\ast$ with $\norm[X^\ast]{φ}\leq 1$, for any solution $(u,v)$ of \eqref{sys} for $D\in\Cdm$ and $(u_0,v_0)\in\ID$ we have 
\begin{align*}
 \frac1{m-1}&\betr{\intnT\io (u^{m-1})_tφ}=\betr{\intnT\io u^{m-2}u_tφ}\\
 &\leq \betr{\intnT\io u^{m-2}φ\na\cdot(D(u)\na u)}+\betr{\intnT\io u^{m-2}φ\na\cdot\kl{\frac{u}v\na v}}\\
 &\leq |m-2|\betr{\intnT\io u^{m-3}φD(u)|\na u|^2} +\betr{\intnT\io u^{m-2}D(u)\na u\cdot\na φ}\\
 &\quad +|m-2|\betr{\intnT\io \frac{u^{m-2}φ}v\na v\cdot\na u}+\betr{\intnT\io \frac{u^{m-1}}v\na v\cdot\na φ}=:I_1+I_2+I_3+I_4, 
\end{align*}
where we can estimate \(I_1\leq|m-2| c_8c_9\), 
\begin{align*}
I_2\leq  \frac12\intnT\io u^{m-3}D(u)|\na u|^2+\frac12\intnT\io u^{m-1}D(u)|\na φ|^2 \leq \frac{c_8}2+\frac12c_6^{m-1}c_{10}^2 \sup_{s\in[0,c_6]}D(s),
\end{align*}
moreover 
\begin{align*}
I_3\leq c_9c_4|m-2|\norm[L^1(\Om\times(0,T))]{u^{m-2}\na u} 
\leq c_9c_4|m-2|\sqrt{|\Om|T}\norm[L^2(\Om\times(0,T))]{u^{m-2}\na u}\\
= \frac{c_9c_4|m-2|\sqrt{|\Om|T}}{m-1}\norm[L^2(\Om\times(0,T))]{\na u^{m-1}} 
\leq \frac{c_9c_4c_7|m-2|\sqrt{|\Om|T}}{m-1}
\end{align*}
and \(I_4\leq c_6^{m-1}c_4c_{11}\), so that finally 
\[
 \norm[L^1((0,T);(W_0^{1,N+1}(\Om))^\ast)]{(u^{m-1})_t}\leq c_{12}+c_{13}\sup_{s\in[0,c_6]}D(s),
\]
where $c_{12}:= c_8c_9|m-2|+\frac{c_8}2+\frac{c_9c_4c_7|m-2|\sqrt{|\Om|T}}{m-1}+c_6^{m-1}c_4c_{11}$ and $c_{13}:=\frac12c_6^{m-1}c_{10}^2$, holds for any solution $(u,v)$ of \eqref{sys} for any $(u_0,v_0)\in\ID$ and any $D\in \Cdm$.
\end{proof}

\begin{proof}[Proof of Theorem \ref{thm:nondegenerate}]
 Lemma \ref{lem:bounds} \ref{bounds:exist} together with \eqref{bd:u} contains Theorem \ref{thm:nondegenerate}.
\end{proof}

\section{Weak solutions in the degenerate case. Proof of Theorem \ref{thm:degenerate}}\label{sec:weaksol}

If the diffusion becomes degenerate at points where $u=0$, we can no longer hope for classical solutions. Therefore we introduce the following definition of weak solutions that are -- in line with our goal of finding solutions that do not blow up in finite time -- locally bounded.

\begin{defn}\label{def:weaksol} 
 Let $δ>0$, $m\geq 1$ and $D\in \Cdm$ and define $\Dbar(s):=\int_0^s D(σ)dσ$ for $s\in [0,∞)$. Moreover, let $(u_0,v_0)$ be as in \eqref{initcond}. By a \emph{locally bounded global weak solution} to \eqref{sys} we mean a pair of functions $(u,v)\colon\Om\times[0,∞)\toℝ^2$ such that 
\begin{align*}
 u&\in L^\infty_{loc}([0,∞);L^\infty(\Om))\\
 \Dbar(u)&\in L^2_{loc}([0,∞);W^{1,2}(\Om))\\
 v&\in L^\infty_{loc}([0,∞);W^{1,∞}(\Om))
\end{align*}
and for every $φ\in C_0^\infty(\Ombar\times[0,∞))$ we have 
\begin{equation}\label{weaku}
 -\intninf\io uφ_t - \io u_0φ(\cdot,0) = -\intninf\io \na \Dbar(u)\cdot \naφ + \intninf\io \frac{u}{v}\na v \cdot \na φ
\end{equation}
and 
\begin{equation}\label{weakv}
 -\intninf\io vφ - \io v_0φ(\cdot,0) = -\intninf\io \na v\cdot \na φ - \intninf \io uv φ.
\end{equation}
\end{defn}


Having prepared a lot of bounds on solutions to \eqref{sys} for $D\in \Cdm^+$ that are uniform in $D\in \Cdm^+$ (Lemma \ref{lem:bounds}), we approximate $D\in \Cdm$ and find a limit of the corresponding solutions.    

\begin{proof}[Proof of Theorem \ref{thm:degenerate}]
 Let $D\in \Cdm$. For any $ε>0$ we define $D_{ε}(s):=D(s+ε)$, $s\in[0,∞)$, and note that, for any $ε>0$, $D_{ε}\in\Cdm^+$. We choose $(\une,\vne)\in (C^1(\Ombar))^2$ such that $\une\to u_0$ and $\vne\to v_0$ in $L^1(\Om)$ as $ε\downto0$ and that there is $K>0$ such that for all $ε\in(0,1)$ we have $\norm[\Lmme]{\une}+\norm[W^{1,\infty}(\Om)]{\vne}\leq K$ and $\norm[\Liom]{\vne}>\frac1K$, and let $(u_{ε},v_{ε})\in(\usualspace[∞])^2$ denote a solution to 
\begin{equation}\label{syseps}
\begin{cases} 
 \quad u_{εt}=\na\cdot(D_{ε}\na u_{ε}) - \na \cdot \kl{\frac{u_{ε}}{v_{ε}}\na v_{ε}} & \text{in } \Om\times(0,∞),\\
 \quad v_{εt}=\Delta v_{ε} - u_{ε}v_{ε}& \text{in } \Om\times(0,∞),\\
 \;\,u_{ε}(\cdot,0)=\une, \quad v_{ε}(\cdot,0)=\vne& \text{in } \Om,\\
 \delny u_{ε}\amrand = 0 = \delny v_{ε}\amrand&\text{in } (0,\infty),
\end{cases}
\end{equation}
which exists due to \ref{lem:bounds} \ref{bounds:exist}. \\
Let us define $\Dbar_{ε}(s):=\int_0^s D_{ε}(σ)dσ$, $s\in[0,∞)$. We claim that for every $n\inℕ$ there is a sequence $(ε_{n,k})_{k\inℕ}$ such that $ε_{n,k}\to 0$ as $k→∞$ for any $n\inℕ$, that for $n>1$ the sequence $(ε_{n,k})_{k\inℕ}$ is a subsequence of $(ε_{n-1,k})_{k\inℕ}$ and that for any $n\inℕ$ 
\begin{equation}\label{eq:convergence}
\begin{cases}
 u_{ε_{n,k}} &\text{ converges a.e. in } \Om\times(0,n) \text{ and in } L^1(\Om\times(0,n))\\
 \Dbar_{ε_{n,k}}(u_{ε_{n,k}}) &\text{ converges weakly in } L^2((0,n);W_0^{1,2}(\Om))\\
 v_{ε_{n,k}} &\text{ converges uniformly in } \Om\times(0,n)\\
 \na v_{ε_{n,k}} &\text{ converges weakly$^\ast$ in } L^{∞}(\Om\times(0,n))\\
 \frac1{v_{ε_{n,k}}}\na v_{ε_{n,k}} &\text{ converges weakly$^\ast$ in } L^{∞}(\Om\times(0,n))
\end{cases}
\end{equation}
as $k→∞$. For $n=0$ we choose an arbitrary monotone sequence $(ε_{0,k})_{k\inℕ}\subset(0,1)$ which converges to $0$. Let $n\inℕ$ and let us assume that some sequence $(ε_{n-1,k})_{k\inℕ}$ with properties as in \eqref{eq:convergence} is given. Then by Lemma \ref{lem:bounds} \ref{bounds:bounds}, more precisely, by \eqref{bd:u}, there is $c_1(n)>0$ such that 
\begin{equation}\label{bd:def:c1}
 \norm[L^{∞}(\Om\times(0,n))]{u_{ε_{n-1,k}}}\leq c_1(n) \qquad \text{for all } k\in ℕ.
\end{equation}
We abbreviate 
\[
 d_n:= \sup_{ε\in(0,1)} \sup_{0\leq s\leq c_1(n)} D_{ε}(s)\leq \sup_{0\leq s\leq c_1(n)+1} D(s).
\]
Then 
\[
 \Dbar_{ε_{n-1,k}}(u_{ε_{n-1,k}}(x,t))\leq \int_0^{c_1(n)} d_n = c_1(n)d_n \quad \text{ for all } (x,t)\in\Om\times(0,n)
\]
and combining this with \eqref{bd:Dunau}, we find $c_2(n)>0$ such that for all $k\inℕ$ 
\[
 \norm[L^2((0,n);W^{1,2}(\Om))]{\Dbar_{ε_{n-1,k}}(u_{ε_{n-1,k}})}\leq c_2(n). 
\]
Hence there is a subsequence $(ε_{n,k}^{(1)})_{k\inℕ}$ of $(ε_{n-1,k})_{k\inℕ}$ such that $(\Dbar_{ε_{n,k}^{(1)}}(u_{ε_{n,k}^{(1)}}))_{k\inℕ}$ is weakly convergent in $L^2((0,n);W^{1,2}(\Om))$. Moreover, \eqref{bd:naumme}, \eqref{bd:u} and \eqref{bd:ut} show that there is $c_3(n)$ such that for all $k\in ℕ$
\[
 \norm[L^2((0,n);W^{1,2}(\Om))]{u_{ε_{n,k}^{(1)}}^{m-1}}\leq c_3(n), \qquad \norm[L^1((0,n);(W_0^{1,N+1})^\ast)]{\kl{u_{ε_{n,k}^{(1)}}^{m-1}}_t}\leq c_3(n).
\]
Since $W^{1,2}(\Om)\cptembeddedinto \Lom2 \embeddedinto (W_0^{1,N+1}(\Om))^\ast$, we can invoke a version of the Aubin-Lions lemma (\cite[Cor. 8.4]{simon}) to find a subsequence $({ε_{n,k}^{(2)}})_{k\inℕ}$ of $({ε_{n,k}^{(1)}})_{k\inℕ}$ such that $(u_{ε_{n,k}^{(2)}}^{m-1})_{k\inℕ}$ is convergent in $L^2((0,n);L^2(\Om))$, and a further subsequence $({ε_{n,k}^{(3)}})_{k\inℕ}$ of $({ε_{n,k}^{(2)}})_{k\inℕ}$ such that $(u_{ε_{n,k}^{(3)}}^{m-1})_{k\inℕ}$ and thus, by continuity of $[0,∞)\ni x\mapsto x^{\frac1{m-1}}$, also $(u_{ε_{n,k}^{(3)}})_{k\inℕ}$ converge a.e. in $\Om\times(0,n)$ as well as with respect to the norm of $L^1(\Om\times(0,n))$ due to Lebesgue's dominated convergence theorem and the fact that the constant $c_1(n)$ is integrable over $\Om\times(0,n)$.\\
Moreover, \eqref{bd:v} and \eqref{bd:vt} ensure the existence of $c_4(n)>0$ such that 
\[
 \norm[L^{∞}((0,n);W^{1,∞}(\Om))]{v_{ε_{n,k}^{(3)}}}\leq c_4(n), \qquad \norm[L^{2}((0,n);(W_0^{1,1}(\Om))^{*})]{\kl{v_{ε_{n,k}^{(3)}}}_t}\leq c_4(n)\qquad \text{for all } k\inℕ
\]
and again due to $W^{1,∞}(\Om)\cptembeddedinto C^0(\Ombar)\embeddedinto (W_0^{1,1}(\Om))^\ast$ and \cite[Cor. 8.4]{simon} we find a subsequence $(ε_{n,k}^{(4)})_{k\inℕ}$ of $(ε_{n,k}^{(3)})_{k\inℕ}$ such that $(v_{ε_{n,k}^{(4)}})_{k\inℕ}$ converges uniformly in $\Om\times(0,n)$. Additionally, \eqref{bd:v} produces another subsequence $(ε_{n,k}^{(5)})_{k\inℕ}$ of $(ε_{n,k}^{(4)})_{k\inℕ}$ such that $(\na v_{ε_{n,k}^{(5)}})_{k\inℕ}$ converges weakly$^\ast$ in $L^\infty(\Om\times(0,n))$. Finally, owing to the bound in \eqref{bd:nav/v}, we can extract a further subsequence $(ε_{n,k})_{k\inℕ}$ of $(ε_{n,k}^{(5)})_{k\inℕ}$ such that also $\kl{\frac1{v_{ε_{n,k}}}\na v_{ε_{n,k}}}_{k\inℕ}$ is weakly$^\ast$ convergent in $L^\infty(\Om\times(0,n))$.\\
We then use the diagonal sequence $(\epstilde_k)_{k\inℕ}:=(ε_{k,k})_{k\inℕ}$ to find functions $u,v,z\colon\Om\times[0,∞)\toℝ$ and $ζ,ξ\colon\Om\times[0,∞)\toℝ^N$ such that 
\begin{align}
 u_{\epstilde_k} \to u \qquad&\text{ in } L^1_{loc}([0,∞),\Lom1) \text{ and a.e. in } \Om\times(0,∞), \label{con:u}\\
 v_{\epstilde_k}\to v \qquad&\text{in } L^{\infty}_{loc}([0,∞);C^0(\Ombar)),\label{con:v}\\
 \Dbar_{\epstilde_k}(u_{\epstilde_k})\wto z \qquad&\text{in }L^2_{loc}([0,∞);W^{1,2}(\Om)),\label{con:Dbar}\\
 \na v_{\epstilde_k}\weakstarto ζ \qquad&\text{in } L^\infty_{loc}([0,∞),L^\infty(\Om)) \text{ and } \label{con:nav}\\
 \frac1{v_{\epstilde_k}}\na v_{\epstilde_k}\weakstarto ξ\qquad&\text{in } L^\infty_{loc}([0,∞),L^\infty(\Om))\label{con:nav/v}
\end{align}
as $k→∞$. Since $u_{\epstilde_k}+\epstilde_k\to u$ a.e. and $\Dbar$ is continuous, also $\Dbar_{\epstilde_k}(u_{\epstilde_k})=\Dbar(u_{\epstilde_k}+\epstilde_k)-\Dbar(\epstilde_k)\to \Dbar(u)-\Dbar(0)=\Dbar(u)$ a.e., 
and hence $z=\Dbar(u)$. Also, \eqref{con:v} and \eqref{con:nav} imply $ζ=\na v$ and the combination of \eqref{con:v} and \eqref{con:nav/v} shows that $ξ=\frac1v\na v$. \\
We let $φ\in C_0^{\infty}(\Om\times[0,∞))$. Then \eqref{syseps} entails that 
\[
 -\intninf\io u_{\epstilde_k}φ_t - \io u_{0\epstilde_k} φ(\cdot,0) = -\intninf\io \na \Dbar_{\epstilde_k}(u_{\epstilde_k})\cdot \naφ + \intninf\io \frac{u_{\epstilde_k}}{v_{\epstilde_k}}\na v_{\epstilde_k} \cdot \na φ
\]
and 
\[
 -\intninf\io v_{\epstilde_k}φ - \io v_{0\epstilde_k} φ(\cdot,0) = -\intninf\io \na v_{\epstilde_k} \cdot \na φ - \intninf \io u_{\epstilde_k}v_{\epstilde_k} φ,
\]
so that passing to the limit as $k→∞$ in each of these integrals shows that $(u,v)$ satisfies \eqref{weaku} and \eqref{weakv}. 
That $u\in L^\infty_{loc}([0,∞);L^\infty(\Om))$ 
is also entailed by \eqref{con:u} and \eqref{bd:def:c1}. Hence $(u,v)$ is a locally bounded global weak solution to \eqref{sys} in the sense of Definition \ref{def:weaksol}. 
\end{proof}

\section{Acknowledgement}
The author acknowledges support of the {\em Deutsche Forschungsgemeinschaft} within the project {\em Analysis of chemotactic cross-diffusion in complex frameworks}. 

{\footnotesize
\def\cprime{$'$}

}
\end{document}